\documentclass[11pt, abstracton, letterpaper]{scrartcl}
\usepackage{etex}
\usepackage{lmodern}
\usepackage[in]{fullpage}
\usepackage[utf8x]{inputenc} 
\usepackage[T1]{fontenc}    
\usepackage[english]{babel}
\usepackage{amsmath}
\usepackage{amsfonts}
\usepackage{amssymb}
\usepackage{amsthm}
\usepackage{mathrsfs}
\usepackage{mathtools}
\usepackage[all,cmtip]{xy}
\usepackage{tikz} \usetikzlibrary{arrows,chains,matrix,positioning,scopes}
\usepackage{tikz-cd}
\usepackage{prettyref}
\usepackage{setspace}
\usepackage[unicode, pdftitle={Quiver D-Modules and the Riemann-Hilbert Correspondence}, pdfauthor={Stephanie Zapf}, colorlinks=true, citecolor=black,%
pdfstartview={FitV}, linkcolor=black, urlcolor=black, bookmarks=false]{hyperref}
\usepackage[numbers,square]{natbib} 

\newtheoremstyle{definition}{1.0em}{1.0em}{\itshape}{}{\bfseries}{.}{0.5em}{}
\newtheoremstyle{satz}{1.0em}{0.5em}{\itshape}{}{\bfseries}{.}{0.5em}{}

\theoremstyle{definition}
\newtheorem{defn}{Definition}[section]

\theoremstyle{satz}
\newtheorem{prop}[defn]{Proposition}
\newtheorem{thm}[defn]{Theorem}
\newtheorem{lem}[defn]{Lemma}
\newtheorem{cor}[defn]{Corollary}
\newtheorem{propdefn}[defn]{Proposition/Definition}
\newtheorem{defnprop}[defn]{Definition/Proposition}

\newrefformat{defn}{Definition~\ref{#1}}
\newrefformat{propdefn}{Proposition/Definition~\ref{#1}}
\newrefformat{defnprop}{Definition/Proposition~\ref{#1}}
\newrefformat{thm}{Theorem~\ref{#1}}
\newrefformat{prop}{Proposition~\ref{#1}}
\newrefformat{lem}{Lemma~\ref{#1}}
\newrefformat{cor}{Corollary~\ref{#1}}
\newrefformat{sec}{Section~\ref{#1}}
\newrefformat{alg}{(ALG)}

\DeclareRobustCommand*{\Cn}{$\mathbb{C}^n$}
\DeclareRobustCommand*{\Dheading}{{$\mathscr{D}$}}

\renewcommand{\C}{\mathbb{C}}%
\newcommand{\N}{\mathbb{N}}
\newcommand{\R}{\mathbb{R}}
\newcommand{\Z}{\mathbb{Z}}

\renewcommand{\a}{\! \! \! \!}
\newcommand{\set}[1]{\{ #1 \}}
\newcommand{\sbt}{{\,\begin{picture}(-1,1)(-1,-3)\circle*{2.5}\end{picture}\ }}
\newcommand{\xfrac}[2]{%
\mbox{\raisebox{0.3ex}{\ensuremath{\displaystyle #1}\hspace{-0.3ex}}%
{\raisebox{-0.1ex}{\Large /}\hspace{-0.4ex}}%
\raisebox{-0.8ex}{\ensuremath{\displaystyle #2}}}}

\newcommand{\D}{\mathscr{D}}
\newcommand{\M}{\mathscr{M}}
\newcommand{\V}{\mathcal{V}}
\newcommand{\gr}{\operatorname{gr}}

\newcommand{\can}{\operatorname{can}}
\newcommand{\var}{\operatorname{var}}

\newcommand{\Qui}{\mathcal{Q}ui}
\newcommand{\Quis}{\Qui^{\Sigma_1}}
\newcommand{\Cc}{\mathcal{C}}

\newcommand{\Id}{\operatorname{Id}}
\newcommand{\Hom}{\operatorname{Hom}}
\newcommand{\im}{\operatorname{im}}
\renewcommand{\Im}{\operatorname{Im}}
\renewcommand{\Re}{\operatorname{Re}}
\newcommand{\Spec}{\operatorname{Spec}}
\newcommand{\mult}{\operatorname{mult}}
\newcommand{\length}{\operatorname{length}}
\newcommand{\Oh}[1]{\mathcal{O}_{ #1 }}

\allowdisplaybreaks
\tikzset{>=stealth',every on chain/.append style={join}, tip/.style={>->}, every join/.style={->}}
\begin{document}

\title{\Large Quiver $\mathcal D$-Modules and the Riemann-Hilbert Correspondence}
\author{Stephanie Zapf}
\date{\vspace{-3ex}}
\maketitle
\begin{abstract}
In this paper, we show that every regular singular $\D$-module in $\C^n$ whose singular locus is a normal crossing is isomorphic to a quiver 
$\D$-module -- a $\D$-module whose definition is based on certain representations of the hypercube quiver. To be more precise we give an equivalence of the respective categories. Our definition 
of quiver $\D$-modules 
is based on the one of Khoroshkin and Varchenko. To prove the equivalence, we use an equivalence by Galligo, Granger and Maisonobe for regular 
singular $\D$-modules whose singular locus is a normal crossing which involves the classical Riemann-Hilbert correspondence.
\end{abstract}
\vskip 2ex

The classical version of the Riemann-Hilbert correspondence, as it was proven independently by Masaki Kashiwara \cite{Kashiwara-RH} and Zoghman 
Mebkhout \cite{Mebkhout-RH} in 1980, yields an equivalence \linebreak \mbox{$ \mathcal{M}od_{\text{rh}} (\D_X) \overset{\cong}{\longrightarrow}  \text{Perv}(X)$}~between regular holonomic 
$\D_X$-modules and perverse sheaves on $X$. In dimension one locally at 0 other equivalences are known which one might obtain 
from the classical version: The equivalence of the category $\mathcal{M}od_{\text{rh}} (\D)$ of regular 
holonomic $\D$-modules with the category $\Cc_1$ and the category $\Quis_1$, the categories
of finite quiver representations $\begin{tikzcd}[column sep=normal]
E  \arrow[yshift=0.5ex, sloped]{r}[above]{u} & 
F  \arrow[yshift=-0.5ex, sloped]{l}[below]{v} 
\end{tikzcd} \a $
over $\C$ fulfilling that $u \circ v + \Id$ is invertible and $\Spec (u \circ v)  \subset \Sigma_1$, respectively (see \cite{Malgrange}, 
\cite{Bjoerk} or \cite{Dimca}). In higher dimension André Galligo, Michel Granger and Philippe Maisonobe proved that in the case of a normal crossing 
divisor in dimension~$n$, the category of perverse sheaves with respect to the induced normal crossing stratification is equivalent to the category $\Cc_n$ 
(the generalization of $\Cc_1$). This means, using the Riemann-Hilbert correspondence, that the category of regular holonomic $\D$-modules whose singular 
locus is a normal crossing is equivalent to~$\Cc_n$ (see \cite{GGM1} and \cite{GGM2}). However, it is not that easy to assign a $\D$-module to a given 
quiver representation with respect to this equivalence concretely.

A contribution to the question of how to assign to a quiver representation a $\D$-module comes from Sergei Khoroshkin and Alexander Varchenko 
\cite{KV}. To a given hyperplane arrangement in~$\C^n$, they associate a quiver. And to each finite representation over $\C$ of such a 
quiver, they associate a $\D$-module in a rather intuitive way. This yields a functor $E$ from the category of representations over these quivers into the 
category of holonomic $\D$-modules. Using this definition in dimension one, one sees that this gives a functor from $\Quis_1$ to $\mathcal{M}od_{\text{rh}} 
(\D)$ one can use as quasi-inverse for the equivalence above. In particular one sees that every regular holonomic $\D$-module in dimension one locally at 0 
is isomorphic to a quiver $\D$-module.  This makes their construction very promising for higher dimensions. 

The main idea of our work is to use this construction of quiver $\D$-modules by Khoroshkin and Varchenko in the case of a normal crossing hyperplane 
arrangement and to combine it with the theorem of Galligo, Granger and Maisonobe. In \prettyref{sec:repr_hypercube_quiver} we give some general statements 
on representations of the hypercube quiver. In \prettyref{sec:quiver_dmodules} 
we define the category of quiver $\D$-modules and give their main properties. In our Main \prettyref{thm:main_theorem} we give 
the link between this category and the category $\Quis_n$ using the theorem of Galligo, Granger and Maisonobe.

\vspace{1ex}
\noindent 
\emph{Acknowledgements.} 
First and foremost, I would like to thank my supervisor Marco Hien for introducing me to this subject and supporting me
throughout my research. I am also grateful to Michel Granger for giving me the 
opportunity to do a residence for research in Angers with plenty of discussions and suggestions. Finally, I want to thank Giovanni Morando and Hedwig Heizinger for various helpful discussions.


\section{Finite representations of the hypercube quiver} \label{sec:repr_hypercube_quiver}

In the following let us consider finite representations over $\C$ of the following quiver:\\
Let $n \in \N^+$ and let $\mathcal{P}(\set{1, \ldots ,n })$ denote the power set of $\set{1, \ldots , n}$. The quiver consists of $2^n$ vertices which we 
denote by $I \in \mathcal{P}(\set{1, \ldots ,n }) $, and $n 2^n$ oriented edges\\[-2ex]
\begin{displaymath}
\begin{tikzcd}[column sep=normal]
I \arrow[yshift=0.6ex, sloped]{r}[above]{ } & {I \cup \set{i}} \arrow[yshift=-0.6ex, sloped]{l}[below, inner sep=0.8ex]{ }
\end{tikzcd}\vspace*{-1ex}
\end{displaymath}
for $i \in \set{1 , \ldots , n} \setminus I$. This gives us a kind of hypercube quiver by imaging that the vertices of the quiver lie on the vertices of 
a $n$-dimensional hypercube, and we have two edges exactly for every edge of the hypercube.


\subsection{Definitions and basic properties}

In the following we are going to define three standard categories of hypercube quiver representations, denoted $\Qui_n$, $\Cc_n$ and $\Quis_n$. Let us 
start with the definition of $\Qui_n$. This is basically just the category of finite representations over $\C$ of the above hypercube quiver.

\begin{defn} \label{defn:Qui_n}
The category $\Qui_n$ for $n \in \N^+$ is defined as follows:
\begin{itemize}
\item The objects consist of $2^n$ finitely generated $\C$-vector spaces denoted $V_I$ where $I \in \mathcal{P}(\set{1, \ldots ,n })$, equipped 
with $n2^{n}$ linear maps $u_{I , i}$ and $y_{I, i}$  for $i \in \set{1 , \ldots , n} \setminus I$,
\begin{displaymath}
\begin{tikzcd}[column sep=normal]
V_I \arrow[yshift=0.6ex, sloped]{r}[above]{u_{I ,i} } & V_{I \cup \set{i}}  \arrow[yshift=-0.6ex, sloped]{l}[below, inner sep=0.8ex]{y_{I, i} }
\end{tikzcd}
\end{displaymath}
and they satisfy the following commutativity conditions for $i,j \in \set{1, \ldots , n} \setminus I$:
\begin{gather*}
u_{I \cup \set{i}, j} \circ u_{I ,i} = u_{I \cup \set{j}, i} \circ u_{I ,j}   \ \ \ \ \ \ y_{I ,i} \circ y_{I \cup \set{i}, j}  =   y_{I ,j} \circ y_{I 
\cup \set{j}, i}  \\
y_{I \cup \set{i}, j} \circ u_{I \cup \set{j},i} = u_{I,i} \circ y_{I ,j} 
\end{gather*}
\item A morphism between two objects $(V_I, u_{I,i}, y_{I,i})$ and $(V^\prime_I, u^\prime_{I,i}, y^\prime_{I,i})$ is given by $2^n$ linear maps 
$h_I \colon V_I \rightarrow V^\prime_I$ such that $u^\prime_{I,i} \circ h_I =  h_{I \cup \set{i}} \circ u_{I,i}$ and $ h_I \circ y_{I,i} = 
y^\prime_{I,i} \circ h_{I \cup \set{i}}$ for $i \in \set{1 , \ldots , n} \setminus I$.
\end{itemize}
\end{defn}
Now, let us define the categories $\Cc_n$ and $\Quis_n$. They are full subcategories of $\Qui_n$ fulfilling an additional constraint on their objects.  
\begin{defn}
The category $\Cc_n$ is the full subcategory of $\Qui_n$ such that every object $(V_I, u_{I,i}, y_{I,i})$ additionally fulfils that $u_{I,i} 
\circ y_{I,i} + \Id_{V_{I \cup \set{i}}}$ and $y_{I,i} \circ u_{I,i} + \Id_{V_I}$ are invertible.
\end{defn}
\begin{defn}
The category $\Quis_n$ is the full subcategory of $\Qui_n$ such that every object $(V_I, u_{I,i}, y_{I,i})$ additionally fulfils that $\Spec (u_{I,i} \circ y_{I,i}) , \, \Spec (y_{I,i} \circ u_{I,i} 
) \subset \Sigma_1 \vcentcolon= \Sigma +1$ where \begin{displaymath}
 \Sigma \vcentcolon= \Bigg\{ \alpha \in \C \  \biggr| \, -1 \leq \Re (\alpha) \leq 0, \ \Im (\alpha) = \begin{cases}
                                                                                        \geq 0 & \text{if $ \Re (\alpha) = -1$,}\\
										    < 0 & \text{if $\Re (\alpha) = 0$,} \\
										    \text{arbitrary} & \text{otherwise.}
                                                                                       \end{cases} \ \ \Bigg\} 
\end{displaymath}
\end{defn}
We note that $ u_{I,i} 
\circ y_{I,i} + \Id_{V_{I \cup \set{i}}}$ is invertible iff $y_{I,i} \circ u_{I,i} + \Id_{V_I}$ is invertible, and  $\Spec ( u_{I,i} \circ y_{I,i} ) 
\subset \Sigma_1 $ iff $\Spec ( y_{I,i} \circ u_{I,i} ) \subset \Sigma_1$.\\

The last topic in this subsection is a dualizing functor acting on our quiver categories.

\begin{defnprop} 
The contravariant functor $D \colon \Qui_n \rightarrow \Qui_n$ is defined on objects $(V_I, u_{I,i}, y_{I,i})$ of $\Qui_n$ by 
\begin{displaymath}
D \left( \! \! \!
\begin{tikzcd}[column sep=normal]
V_I \arrow[yshift=0.6ex, sloped]{r}[above]{u_{I ,i} } & V_{I \cup \set{i}} \arrow[yshift=-0.6ex, sloped]{l}[below, inner sep=0.8ex]{y_{I, i} }
\end{tikzcd}
\! \! \right) \vcentcolon= 
\begin{tikzcd}[column sep=normal]
V^\ast_I \arrow[yshift=0.6ex, sloped]{r}[above]{y^\ast_{I, i} } & V^\ast_{I \cup \set{i}} \arrow[yshift=-0.6ex, sloped]{l}[below, inner sep=0.8ex]{ 
u^\ast_{I ,i}}
\end{tikzcd}
\end{displaymath}
where $V^\ast_I$ is the dual vector space of $V_I$ and $u^\ast_{I,i}$, $y^\ast_{I,i}$ are the dual/transpose maps of $u_{I,i}$, $y_{I,i}$. 
Let $(h_I)$ denote a morphism in $\Qui_n$. Then we set
\begin{displaymath}
 D((h_I)) \vcentcolon= (h_I^\ast)
\end{displaymath}
where $h_I^\ast$ is the dual map of $h_I$. This yields an equivalence of categories where $D$ is its own quasi-inverse, and it also establishes an equivalence from $\Quis_n$ to $\Quis_n$, and from 
$\Cc_n$ to $\Cc_n$ as well.
\end{defnprop}


\subsection{An equivalence of categories}

The goal of this subsection is to prove in \prettyref{thm:C_equivalent_Quis} an equivalence (or rather an isomorphism) of the categories 
$\Cc_n$ and $\Quis_n$. This is a principal component of the present work. We use the following pair of functors:

\begin{defn}
The covariant functors $ Q \colon \Quis_n \rightarrow \Cc_n $ and $\mathcal{G} \colon \Cc_n \rightarrow \Quis_n$ are defined by:\\
Let $(V_I, u_{I,i}, c_{I,i})$ denote an object in $\Quis_n$ and let $(h_I)$ denote a morphism in $\Quis_n$. We set
\begin{align*}
& Q ( (V_I, u_{I,i}, c_{I,i}) ) \vcentcolon=  (V_I, u_{I,i}, y_{I,i})  \ \text{ and } \ Q ((h_I) ) \vcentcolon= (h_I) \\
& \text{where } y_{I,i} \vcentcolon= \sum_{k=1}^\infty \frac{(2 \pi i)^k}{k!} (c_{I,i} \circ u_{I,i})^{k-1} \circ c_{I,i} = c_{I,i} \circ 
\sum_{k=1}^\infty \frac{(2 \pi i)^k}{k!} ( u_{I,i}  \circ c_{I,i})^{k-1}.
\end{align*}
Let $(V_I, u_{I,i}, w_{I,i})$ denote an object in $\Cc_n$ and let $(h_I)$ denote a morphism in $\Cc_n$. Then, set
\begin{align*}
& \mathcal{G} ( (V_I, u_{I,i}, w_{I,i}) ) \vcentcolon= (V_I, u_{I,i}, x_{I,i}) \ \text{ and } \ \mathcal{G} ( (h_I) ) \vcentcolon= (h_I) .
\end{align*}
The map $x_{I,i}$ is given as follows: Let $s_{I,i} \colon V_I \rightarrow V_I$ denote the unique linear map with eigenvalues in $\Sigma_1$ such that
\begin{displaymath}
e^{2 \pi i s_{I,i}} = w_{I,i} \circ u_{I,i} + \Id_{V_I} \ \ \text{ and set } \ \
 x_{I,i} \vcentcolon= \left( \sum_{k=1}^\infty \frac{(2 \pi i )^k}{k!} s_{I,i}^{k-1} \right)^{-1} \circ w_{I,i}.
\end{displaymath}
\end{defn}

\begin{thm} \label{thm:C_equivalent_Quis}
The category $\Cc_n$ is isomorphic to the category $\Quis_n$ using the pair of covariant functors $ Q \colon \Quis_n \rightarrow \Cc_n $ and $ 
\mathcal{G} \colon \Cc_n \rightarrow \Quis_n$.%
\end{thm}%
Note that $i$ in $2\pi i$ is the imaginary unit. Before proving the theorem, we verify two helpful statements from matrix analysis.

\begin{prop} \label{prop:matrix_ln_unique}
 Let $E$ denote a finite dimensional $\C$-vector space and let $f \colon E \rightarrow E$ denote an invertible linear map. Then there exists a unique 
linear map $g \colon E \rightarrow E$ with $\Spec (g) \subset \Sigma_1$ such that
\begin{displaymath}
f = e^{2 \pi i g} .  
\end{displaymath}
\end{prop}

\begin{proof}
We choose the branch of the logarithm defined on $\C \setminus  \R_{\geq 0}$ with image contained in $\{ \alpha \in \C \mid 0 < \Im ( \alpha) < 
2 \pi \}$, and use a unique extension to $\C \setminus \{ 0 \}$ with image in $2 \pi i \Sigma_1$. Note that every complex number has a unique 
representative in this strip up to $\pm 2 \pi i \N$. The existence and uniqueness of $2 \pi i g$ follows now with the aid of \cite[Corol\-lary 
6.2.12]{Horn-Johnson} which deals with finding a matrix~$A$ for a given invertible matrix $B$ such that $e^A = B$. 
\end{proof}

The next corollary will be auxiliary to prove commutativity conditions later.

\begin{cor} \label{cor:c_expa=expb_c_iff_ca=bc}
Let $E$, $F$ denote two finite dimensional $\C$-vector spaces and let $\gamma \colon E \rightarrow F$ denote a linear map. Furthermore, let $\alpha \colon E 
\rightarrow E$ and $\beta \colon F \rightarrow F$ denote two linear maps with $\Spec( \alpha), \Spec (\beta) \subset \Sigma_1$. Then:
\begin{displaymath}
 \gamma \circ e^{2 \pi i \alpha} = e^{2 \pi i \beta} \circ \gamma \ \ 
\Longleftrightarrow \ \ \gamma \circ \alpha = \beta \circ \gamma 
\end{displaymath}
\end{cor}

\begin{proof}
We only need to prove the direction ``$\Rightarrow$''. This proof is divided into three parts:
\begin{itemize}
 \item [(i)] Assume that $\gamma \colon E \rightarrow F$ is invertible. We receive 
\begin{displaymath}
\gamma \circ e^{2\pi i \alpha} = e^{2 \pi i \beta } \circ \gamma \ \ 
\Longleftrightarrow \ \ e^{2 \pi i \beta } = e^{2 \pi i \, ( \gamma \circ \alpha \circ 
\gamma^{-1} ) }.
\end{displaymath}
The eigenvalues of $\beta$ and $\gamma \circ \alpha \circ \gamma^{-1}$ are both contained in $\Sigma_1$. Using the uniqueness given in 
\prettyref{prop:matrix_ln_unique}, we receive the claimed equality.
\item [(ii)] To prove the general case, we need the following small statement:
\par
\begingroup
\leftskip=3em
 $\alpha$ preserves a linear subspace $\tilde{E}$ of E, i.\,e.\ $\alpha (\tilde{E}) \subset \tilde{E}$. $\Leftrightarrow$ $e^{2 \pi i \alpha}$ preserves 
$\tilde{E}$.
\par
\endgroup
The direction ``$\Rightarrow$'' is clear. To prove ``$\Leftarrow$'', use \cite[Corol\-lary 6.2.12]{Horn-Johnson} to receive that $\alpha$ is a polynomial in $e^{2 
\pi i \alpha}$ (the concrete form of the polynomial depends on the map $\alpha$). 
\item [(iii)] Now, let us prove the general case. Examining $\gamma \circ e^{2 \pi i \alpha} = e^{2 \pi i \beta} \circ \gamma$, we see that $e^{2 \pi 
i \beta}$ preserves $\im (\gamma)$ and that $e^{2 \pi i \alpha}$ preserves $\ker(\gamma)$. By part (ii), $\im(\gamma)$ and $\ker(\gamma)$ are preserved by $\beta$ and $\alpha$, respectively.  This 
gives us well-defined maps
\begin{displaymath}
\overline{e^{2 \pi i \alpha}} \colon \xfrac{E}{\ker (\gamma)} \rightarrow \xfrac{E}{\ker (\gamma)} \ \text{ and }  \
\overline{\alpha} \colon \xfrac{E}{\ker (\gamma)} \rightarrow \xfrac{E}{\ker (\gamma)}
\end{displaymath}
where $\overline{e^{2 \pi i \alpha}} = e^{2\pi i \overline{\alpha}}$. Consider the following diagram whose first and last square commute:
 \begin{displaymath}
\begin{tikzcd}[column sep=5.5ex, row sep=4.5ex]
E \arrow[two heads]{r}  \arrow{d}[left]{ \alpha} &  
 \xfrac{E}{\ker (\gamma)} 
\arrow{r}{\overline{\gamma}}
\arrow{d}[left]{\overline{ \alpha}} & 
\operatorname{im} (\gamma) 
\arrow[hook]{r}
\arrow{d}[left]{ \beta} & 
F
 \arrow{d}[left]{ \beta} \\
E \arrow[two heads]{r}  &  
 \xfrac{E}{\ker (\gamma)} 
\arrow{r}{\overline{\gamma}} &
\operatorname{im} (\gamma)
\arrow[hook]{r}  & 
F 
\end{tikzcd} 
\end{displaymath}
One has $\Spec (\overline{\alpha}) \subset \Spec(  \alpha)  \subset  \Sigma_1$. By part (i) we receive now that the commutativity of
 \begin{displaymath}
\begin{tikzcd}[column sep=normal, row sep=normal]
 \xfrac{E}{\ker (\gamma)} 
\arrow{r}{\overline{\gamma}}
\arrow{d}[left]{e^{2 \pi i \overline{ \alpha}}} & 
\im (\gamma) 
 \arrow{d}[left]{e^{2 \pi i \beta}}  \\
 \xfrac{E}{\ker (\gamma)} 
\arrow{r}{\overline{\gamma}} &
\im (\gamma)
\end{tikzcd}
\text{ implies that }
\begin{tikzcd}[column sep=normal, row sep=normal]
 \xfrac{E}{\ker (\gamma)} 
\arrow{r}{\overline{\gamma}}
\arrow{d}[left]{\overline{ \alpha}} & 
\im (\gamma)
 \arrow{d}[left]{ \beta}  \\
 \xfrac{E}{\ker (\gamma)} 
\arrow{r}{\overline{\gamma}} &
\im (\gamma)
\end{tikzcd}
\end{displaymath}
commutes. This yields the commutativity of the above diagram and proves the claim.\qedhere
\end{itemize}
\end{proof}

\begin{proof}[Proof of \prettyref{thm:C_equivalent_Quis}]
For simplicity, we set for a linear map $A$
\begin{displaymath}
\psi(A)\vcentcolon= \sum_{k=1}^\infty \frac{(2\pi i)^k}{k!} A^{k-1} .
\end{displaymath}
We need to check several small statements to obtain the theorem:
\begin{itemize}
\item [(i)] Verify that $x_{I,i}$ is well-defined: The map $w_{I,i} \circ u_{I,i} + \Id_{V_I}$ is invertible by assumption. The existence and uniqueness 
of $s_{I,i}$ follows from \prettyref{prop:matrix_ln_unique}. As $\psi (0) = 2 \pi i \neq 0$ and
\begin{displaymath}
\psi ( \lambda) = \frac{e^{2 \pi i \lambda} -1}{\lambda}\neq 0
\end{displaymath}
for $\lambda \in \Sigma_1 \setminus \set{0}$, we receive that the eigenvalues of $\psi ( s_{I,i})$ are non-zero and therefore $x_{I,i} = 
\psi(s_{I,i})^{-1} \circ w_{I,i}$ is well-defined.

\item [(ii)] We have to check that $Q \colon \Quis_n \rightarrow \Cc_n$ is a well-defined functor: \\ 
Let $(V_I, u_{I,i}, c_{I,i})$ denote an object in $\Quis_n$ and let $(V_I, u_{I,i}, y_{I,i})$ denote its image under $Q$. This is indeed an object in 
$\Cc_n$: The map
\begin{align*}
 y_{I,i} \circ u_{I,i} + \Id_{V_I} =  \psi(c_{I,i} \circ u_{I,i} ) \circ c_{I,i} \circ u_{I,i} + \Id_{V_I} =  e^{ 2 \pi i (c_{I,i} \circ u_{I,i} )}
\end{align*}
is obviously invertible. The commutativity conditions in $\Cc_n$ follow by direct computation from those in $\Quis_n$. 
Now, let  $(h_I)$ denote a morphism from $(V_I, u_{I,i}, c_{I,i})$ to $(V^\prime_I, u^\prime_{I,i}, c^\prime_{I,i})$ in $\Quis_n$. To prove that $Q((h_I)) 
= (h_I)$ is a morphism in $\Cc_n$ from $(V_I, u_{I,i}, y_{I,i})$ to $(V^\prime_I, u^\prime_{I,i}, y^\prime_{I,i})$, we need to check the equations
\begin{align*}
 u^\prime_{I,i} \circ h_I =  h_{I \cup \set{i}} \circ u_{I,i} \ \text{ and } \ h_I \circ y_{I,i} = y^\prime_{I,i} \circ h_{I \cup \set{i}} .
\end{align*}
Both equations follow directly from the fact that $(h_I)$ is a morphism in $\Quis_n$. 

\item [(iii)] We have to check that $\mathcal{G} \colon \Cc_n \rightarrow \Quis_n$ is a well-defined functor:\\
Let $(V_I, u_{I,i}, w_{I,i})$ denote an object in $\Cc_n$ and let $(V_I, u_{I,i}, x_{I,i})$ denote its image under $\mathcal{G}$. $(V_I, u_{I,i}, 
x_{I,i})$ is indeed an object in $\Quis_n$:

\begin{itemize}
 \item [$\bullet$] We have the equality
\begin{align*}
 x_{I,i} \circ u_{I,i} &= \psi( s_{I,i} )^{-1} \circ w_{I,i} \circ u_{I,i} = \psi( s_{I,i} )^{-1} \circ (e^{2 \pi i s_{I,i}} - \Id_{V_I}) = s_{I,i} 
\end{align*}
which gives us $\Spec ( x_{I,i} \circ u_{I,i} ) \subset \Sigma_1$.
 \item [$\bullet$] To prove the commutativity conditions, we use the following identities in $\Cc_n$:
 \begin{align*}
e^{2 \pi i s_{I,j}} \circ w_{I,i} &= w_{I,i} \circ e^{2 \pi i s_{I \cup \set{i},j}}   \\
u_{I,i} \circ e^{2 \pi i s_{I,j}} &= e^{2 \pi i s_{I \cup \set{i},j}} \circ u_{I,i}  \\
 e^{2 \pi i s_{I,i}} \circ e^{2 \pi i s_{I,j}} &=  e^{2 \pi i s_{I,j}} \circ e^{2 \pi i s_{I,i}}
 \end{align*}
These yield with the aid of \prettyref{cor:c_expa=expb_c_iff_ca=bc}:
 \begin{align*}
& \star \    s_{I,j} \circ w_{I,i} = w_{I,i} \circ  s_{I \cup \set{i},j}  & & \Longrightarrow \ w_{I,i} \circ  \psi (s_{I 
\cup \set{i},j} )^{-1} = \psi (s_{I,j})^{-1} \circ w_{I,i}\\
&\star \  s_{I,i} \circ  s_{I,j} =s_{I,j} \circ  s_{I,i} & & \Longrightarrow \
\psi(s_{I,i})^{-1} \circ \psi ( s_{I,j})^{-1} = \psi (s_{I,j})^{-1} \circ \psi ( s_{I,i})^{-1} \\
& \star \  u_{I,i} \circ  s_{I,j} = s_{I \cup \set{i},j} \circ u_{I,i} & & \Longrightarrow \ \psi(s_{I \cup \set{i},j})^{-1} \circ  u_{I,i}  = 
u_{I,i} \circ  \psi (s_{I,j})^{-1}
 \end{align*}
The commutativity conditions follow now immediately.
\end{itemize}

Now, let  $(h_I)$ denote a morphism from 
$(V_I, u_{I,i}, w_{I,i})$ to $(V^\prime_I, u^\prime_{I,i}, w^\prime_{I,i})$ in $\Cc_n$. To prove that $\mathcal{G}((h_I)) = (h_I)$ is a 
morphism from $(V_I, u_{I,i}, 
x_{I,i})$ to $(V^\prime_I, u^\prime_{I,i}, x^\prime_{I,i})$ in $\Quis_n$, we need to verify the identities
\begin{align*}
\star \ & u^\prime_{I,i} \circ h_I =  h_{I \cup \set{i}} \circ u_{I,i}  \\
 \star \ & h_I \circ x_{I,i} = x^\prime_{I,i} \circ h_{I \cup \set{i}} \ \Longleftrightarrow  \
h_I \circ \psi( s_{I,i} )^{-1} \circ w_{I,i} = \psi( 
s^\prime_{I,i} )^{-1} \circ h_I \circ w_{I,i} \, .
 \end{align*}
The first one follows directly. To prove the second equation we use the equality $$ e^{2 \pi i s^\prime_{I,i}} \circ h_I  =
h_I \circ e^{2 \pi i s_{I,i}} .$$
Now, \prettyref{cor:c_expa=expb_c_iff_ca=bc} yields
\begin{displaymath}
 s^\prime_{I,i} \circ h_I  =h_I \circ s_{I,i} \ \text{ and therefore } \ 
 h_I  \circ \psi(s_{I,i})^{-1} = \psi ( s^\prime_{I,i})^{-1} \circ  h_I .
\end{displaymath}

\pagebreak
\item [(iv)] We show that {$Q \circ \mathcal{G} = \Id_{\mathcal{C}_n} $}: For this we need to check for an object $(V_I, u_{I,i}, w_{I,i})$ in $\Cc_n$ that~$(Q \circ \mathcal{G}) ( (V_I, u_{I,i}, 
w_{I,i}) )  = (V_I, u_{I,i}, w_{I,i})$. Let
\begin{align*}
\mathcal{G} ( (V_I, u_{I,i}, w_{I,i}) ) =\vcentcolon  (V_I, u_{I,i}, x_{I,i}) \ &\text{ where } \ x_{I,i} = \psi ( s_{I,i} )^{-1} \circ w_{I,i} 
\ \text{ and}\\
 Q ( (V_I, u_{I,i}, x_{I,i}) ) =\vcentcolon (V_I, u_{I,i}, y_{I,i}) \ &\text{ where } \ y_{I,i} = \psi (x_{I,i} \circ u_{I,i}) \circ x_{I,i} \, 
.
\end{align*}
By part (iii) of the proof, $x_{I,i}$, $u_{I,i}$ and $s_{I,i}$ fulfil the equality $x_{I,i} \circ u_{I,i} = s_{I,i}$. Using the definition of $x_{I,i}$, 
this yields
\begin{displaymath}
 y_{I,i} = \psi (x_{I,i} \circ u_{I,i} ) \circ x_{I,i} = \psi( s_{I,i} ) \circ \psi ( s_{I,i})^{-1} \circ w_{I,i} = w_{I,i} \, .
\end{displaymath}

\item [(iv)]  We show that $\mathcal{G} \circ Q = \Id_{\mathcal{Q}ui^\Sigma_n}$:
We need to verify for an object $(V_I, u_{I,i}, c_{I,i})$ in $\Quis_n$ that~$( \mathcal{G} \circ Q) ( (V_I, u_{I,i}, c_{I,i}) ) 
 = (V_I, u_{I,i}, c_{I,i})$. We set
\begin{align*}
 Q ( (V_I, u_{I,i}, c_{I,i}) ) =\vcentcolon (V_I, u_{I,i}, y_{I,i}) \ &\text{ where } \ y_{I,i} = \psi (c_{I,i} \circ u_{I,i} ) \circ c_{I,i} \ 
\text{ and}\\
 \mathcal{G} ( (V_I, u_{I,i}, y_{I,i}) ) =\vcentcolon (V_I, u_{I,i}, x_{I,i}) \ &\text{ where }  \ x_{I,i} =  \psi ( s_{I,i} )^{-1} \circ 
y_{I,i} \, .
\end{align*}
We have the equality
\begin{displaymath}
e^{2 \pi i (c_{I,i} \circ u_{I,i})} = y_{I,i} \circ u_{I,i} + \Id_{V_I} =  e^{2 \pi i s_{I,i}} .
\end{displaymath}
The eigenvalues of $c_{I,i} \circ u_{I,i}$ and $ s_{I,i}$ are contained in $\Sigma_1$. Thus, the uniqueness of $s_{I,i}$ (cf.~\prettyref{prop:matrix_ln_unique}) yields $c_{I,i} \circ u_{I,i} = 
s_{I,i}$. Hence,
\begin{displaymath}
x_{I,i} =  \psi (s_{I,i} )^{-1} \circ y_{I,i} = 
\psi (c_{I,i} \circ u_{I,i} )^{-1} \circ \psi (c_{I,i} \circ u_{I,i} ) \circ c_{I,i} =  c_{I,i} \, .
\end{displaymath}
 \end{itemize}
All in all, this shows that $Q \colon \Quis_n \rightarrow \Cc_n$ and $\mathcal{G} \colon \Cc_n \rightarrow \Quis_n$ are inverse functors to each other and therefore 
they 
define an isomorphism between the categories $\Quis_n$ and $\Cc_n$.
\end{proof}


\section[\texorpdfstring{Quiver \Dheading-modules in \Cn{} whose singular locus is a normal crossing
}{Quiver D-modules in C\textasciicircum n whose singular locus is a normal crossing
}]{Quiver \Dheading-modules in \Cn{} whose singular locus is a normal crossing 
} \label{sec:quiver_dmodules}

From now on $\Oh{} = \Oh{X}$ will always denote the sheaf of analytic functions on $X= \C^n$ for a fixed integer $n \in \N^+$, and $\D = \D_X$ denotes the sheaf of rings of linear 
partial differential operators with analytic coefficients. Furthermore, we denote by $z_1, \ldots, z_n$ the coordinates of $\C^n$ and by  $\partial_i = 
\frac{\partial}{\partial z_i}$ the $i$-th 
partial derivation operator for $i =1, \ldots, n$.


\subsection{Definitions and basic properties}

Let us first define objects of quiver $\D$-modules. These are $\D$-modules defined in a very natural manner on the basis of certain 
quiver representations where we use the category $\Qui_n$ as starting point. Our definition is based on the one in \cite[Subsection 4.2]{KV} in the case of a normal crossing 
hyperplane arrangement whereas we use analytic $\D$-modules instead of algebraic ones.

\begin{defn}[Variant of \cite{KV}] \label{defn:quiver_dmodule}
Let $\V_n = (V_I, B_{I \cup \set{i}, I}, B_{I, I \cup \set{i}})$ denote an object in the category $\Qui_n$. We define the associated quiver $\D$-module $E \V_n$ as 
the quotient of 
\begin{align*}
\bigoplus_{I \in \mathcal{P}(\set{1, \ldots , n})} \left( \D \otimes_\C \overline{\Omega}_I \otimes_\C V_I \right)
\end{align*}
by the subsheaf $\mathcal{J}$. The sections of $\mathcal{J}$ over $U \subset \C^n$, open, are given by $\C$-linear combinations of the 
following elements 
\begin{align*}
&a \partial_i \otimes_\C \omega_I \otimes_\C v_I - 
a \otimes_\C \omega_{I \cup \set{i}} \otimes_\C B_{I \cup \set{i}, I}(v_I) \ \ \text{and} \\
&a z_i \otimes_\C \omega_{I \cup \set{i}} \otimes_\C v_{I \cup \set{i}} - 
a \otimes_\C \omega_I \otimes_\C B_{I, I \cup \set{i}} (v_{I \cup \set{i}})
\end{align*}
where $I \neq \set{1, \ldots ,n}$, $i \notin I$, $a \in \D(U)$, $v_J \in V_J$ for $J \in \mathcal{P}(\set{1, \ldots , n})$
and $\overline{\Omega}_J \vcentcolon= \set{ c \, \omega_J \mid c \in \C }$ is a 1-dimensional $\C$-vector space generated by the element 
$\omega_J$. The left $\D$-module structure on $E \V_n$ is given by left multiplication.
\end{defn}
$\overline{\Omega}_J$ is used here to clarify to which summand of $\bigoplus_I \left( \D \otimes \overline{\Omega}_I \otimes V_I \right)$ an element 
belongs to. Our next aim is to receive a functor from $\Qui_n$ into the category of $\D$-modules on $\C^n$.

\begin{cor}  \label{cor:mor_Qui_induces_mor_quiver_dmodules}
Let  $\V_n = (V_I, B_{I \cup \set{i}, I}, B_{I, I \cup \set{i}}) $ and $\V^\prime_n = (V^\prime_I, B^\prime_{I \cup \set{i}, I}, B^\prime_{I, I \cup 
\set{i}}) $ denote two objects in $\Qui_n$ and let 
\begin{displaymath}
\phi= (h_I) \colon \V_n \rightarrow \V_n^\prime
\end{displaymath}
denote a morphism from $\V_n$ to $\V_n^\prime$. Then $\phi $ induces naturally a morphism
\begin{displaymath}
 E \phi \colon E\V_n \rightarrow E \V_n^\prime.
\end{displaymath}
\end{cor}

\begin{proof}
 Consider the following diagram whose rows are exact sequences:
\begin{displaymath}
\begin{tikzcd}[column sep=7ex, row sep=7ex]
0 \rar &  
\mathcal{J}
\arrow[hook]{r}
\arrow{d}{\tilde{h}} & 
\bigoplus\limits_{I \in \mathcal{P}(\set{1, \ldots , n})} \left( \D \otimes_\C \overline{\Omega}_I \otimes_\C V_I \right) 
\arrow[two heads]{r} 
\arrow{d}{\bigoplus_I (\Id_\D \otimes \Id_{\overline{\Omega}_I} \otimes h_I)} 
\arrow[color=white]{dl}[color=black, above, xshift=4.5ex, yshift=-2ex, font=\Large]{\circlearrowleft} & 
E \V_n 
\arrow[dashed]{d} 
\rar & 
0\\
0 \rar &  
\mathcal{J}^\prime
\arrow[hook]{r} &
\bigoplus\limits_{I \in \mathcal{P}(\set{1, \ldots , n})} \left( \D \otimes_\C \overline{\Omega}_I \otimes_\C V^\prime_I \right) 
\arrow[two heads]{r}  & 
E \V_n^\prime
\rar & 
0%
\end{tikzcd} 
\end{displaymath}%
$(h_I)$ induces naturally a $\D$-linear map $\tilde{h} \colon \mathcal{J} \rightarrow \mathcal{J}^\prime$ which fulfils on sections over $U \subset \C^n$, 
open,
\begin{align*}
\tilde{h} & \! \left(a \partial_i \otimes \omega_I \otimes v_I - 
a \otimes \omega_{I \cup \set{i}} \otimes B_{I \cup \set{i}, I}(v_I) \right) = \\
&= a \partial_i \otimes \omega_I \otimes h_I(v_I) - a \otimes \omega_{I \cup 
\set{I}}   \otimes h_{I\cup \set{i}} (B_{I \cup \set{i}, I}(v_I)) = \\
&= a \partial_i \otimes \omega_I \otimes h_I(v_I) - a \otimes \omega_{I 
\cup \set{I}}   \otimes B^\prime_{I \cup \set{i}, I}(h_I(v_I))  \hspace*{3em} \text{and}  
\\
\tilde{h} & \! \left( a z_i \otimes \omega_{I \cup \set{i}} \otimes v_{I \cup \set{i}} - 
a \otimes \omega_I \otimes B_{I, I \cup \set{i}} (v_{I \cup \set{i}}) \right) = \\
&  = a z_i \otimes \omega_{I \cup \set{i}} \otimes h_{I \cup 
\set{i}}(v_{I \cup \set{i}}) - 
a \otimes \omega_I \otimes h_I (B_{I, I \cup \set{i}} ( v_{I \cup \set{i}}))  = \\
&= a z_i \otimes \omega_{I \cup \set{i}} \otimes h_{I \cup \set{i}}(v_{I \cup \set{i}}) - 
a \otimes \omega_I \otimes B^\prime_{I, I \cup \set{i}} (h_{I \cup \set{i}}(v_{I \cup \set{i}})).
\end{align*}
This makes the square commute as indicated. In particular, it induces in a natural way a $\D$-linear morphism from $ E\V_n$ to $ E \V_n^\prime$.
\end{proof}

\begin{propdefn}
Let $\mathcal{M}od (\D)$ denote the category of $\D$-modules on $\C^n$. Then we receive a covariant functor, denoted $E$, from the category $\Qui_n$ into 
the category of $\D$-modules 
\begin{displaymath}
E \colon \Qui_n \rightarrow \mathcal{M}od (\D).
\end{displaymath}
$E$ associates to an object $\V_n$ in $\Qui_n$ the object $E \V_n$ in $\mathcal{M}od (\D)$ from \prettyref{defn:quiver_dmodule}, and it associates to a morphism 
$\phi \colon \V_n \rightarrow \V_n^\prime$ in $\Qui_n$ the morphism $E \phi \colon E \V_n \rightarrow E \V_n^\prime$ in $\mathcal{M}od (\D)$ from \prettyref{cor:mor_Qui_induces_mor_quiver_dmodules}. 
The category of quiver \mbox{$\D$-modules} is the essential image of the functor $E$. 
\end{propdefn}
\begin{proof}
Let $(h_I)$ and $(g_I)$ denote two morphisms in $\Qui_n$ with compatible source and target, respectively. Then $E$ preserves the composition of morphisms, 
using \prettyref{cor:mor_Qui_induces_mor_quiver_dmodules}, as
\begin{align*}
(\Id_\D \otimes \Id_{\overline{\Omega}_I} 
\otimes h_I) \circ (\Id_\D \otimes \Id_{\overline{\Omega}_I} \otimes g_I) = (\Id_\D \otimes \Id_{\overline{\Omega}_I} \otimes (h_I \circ g_I)) \ \text{ 
and } \ \widetilde{h \circ g } = \tilde{h} \circ \tilde{g} .
\end{align*} 
As $E$ also preserves the identity morphism, $E$ is indeed a functor from $\Qui_n$ to $\mathcal{M}od(\D)$.
\end{proof}

Now, we define the category $\mathcal{M}od_{\text{rh}}^S (\D)$ of regular singular holonomic $\D$-modules whose singular locus is contained in the 
normal crossing. However, we have to mention that this denomination is a little bit sloppy as in fact the objects in $\mathcal{M}od_{\text{rh}}^S (\D)$ 
need to fulfil a property on their characteristic variety from which the property on the singular locus follows.

\begin{defn} \label{defn:cat_Mod_rh_S}
Let $S \vcentcolon= \{ z_1 \cdot \ldots \cdot z_n = 0 \}$ denote the normal crossing in $\C^n$. $S$ induces naturally a (Whitney) 
stratification of $\C^n$ by $2^n$ disjoint strata $X_I \subset \C^n$ which are defined by
\begin{displaymath}
\overline{X}_I \vcentcolon= \{ z_i = 0 \, \mid \, i \in I \} , \ \ \ 
X_I \vcentcolon= \overline{X}_I \setminus \Bigl( \, \bigcup_{\substack{J \in \mathcal{P} ( \set{1, \ldots , 
n}) \\ \overline{X}_J \subsetneq \overline{X}_I }} \overline{X}_J \Bigr)
\end{displaymath}
for $I \in \mathcal{P} ( \set{1, \ldots , n})$. This fulfils $X_\varnothing = \C^n \setminus S$ and $S= \bigcup_{I, \dim X_I <n} X_I$.\\%
The category $\mathcal{M}od_{\text{rh}}^S (\D)$ is then defined to be the category of regular singular holonomic \mbox{$\D$-modules} whose 
characteristic variety is contained in
\begin{gather*}
 \bigcup_{I \in \mathcal{P} ( \set{1, \ldots , 
n})} T^\ast_{X_I} \C^n \\
\begin{aligned}
\text{where } \ 
 T^\ast_{X_I} \C^n &= \{ (z , \xi ) = (z_1, \ldots, z_n, \xi_1, \ldots, \xi_n ) \in T^\ast \C^n \mid z \in X_I , \,  \xi (v) = 0 \ \forall \ v \in 
T X_I \} = \\
 &= \{ (z_1, \ldots, z_n, \xi_1, \ldots, \xi_n) \in T^\ast \C^n \mid z_i = 0 \Leftrightarrow i \in I, \, \xi_i = 0 \text{ for } 
i \notin I \}.
\end{aligned}
\end{gather*}
\end{defn}

We note that 
\begin{gather*}
  \bigcup_{I \in \mathcal{P} ( \set{1, \ldots , 
n}} T^\ast_{X_I} \C^n  = \Delta_S \\
\text{ where } \  \Delta_S \vcentcolon= \{ (z_1, \ldots, z_n, \xi_1, \ldots, \xi_n ) \in T^\ast \C^n \mid z_i \xi_i = 0 \ \text{for all} \ i 
\in \set{1, \ldots, n} \}
\end{gather*}
which simplifies \prettyref{defn:cat_Mod_rh_S}. Thus, $\mathcal{M}od_{\text{rh}}^S (\D)$ is the category of regular singular holonomic $\D$-modules whose 
characteristic variety is contained in $\Delta_S$.

Let us explain how the property on the singular locus is implied by this: 
Let $\M$ denote a holonomic $\D$-module and let $\operatorname{Char}(\M)$ denote its characteristic variety. The singular locus of $\M$ is defined as the 
projection of $\operatorname{Char}(\M) \setminus T^\ast_{\C^n} \C^n$ onto $\C^n$ where $T^\ast_{\C^n} \C^n$ is the zero section of $T^\ast \C^n$. By direct 
computation one sees that the projection of $\Delta_S \setminus T^\ast_{\C^n} \C^n$ to $\C^n$ is equal to $S = \bigcup_{I, \dim X_I < n} X_I$. Thus, the 
singular locus of $\M$ is contained in $S$.

The stratification $S$ of $\C^n$ in fact also determines the characteristic variety of the quiver $\D$-modules as we will see now.

\begin{thm} \label{thm:E_functor}
The functor $E$ maps from the category $\Qui_n$ into the category $\mathcal{M}od_{\text{rh}}^S (\D)$.
\end{thm}

\begin{proof}
Let $\V_n = (V_J, B_{J \cup \set{j}, J}, B_{J, J \cup \set{j}})$ denote an object in $\Qui_n$. We define the good filtration $F_k E \V_n$ on $E \V_n$ as 
the filtration induced by the exact sequence
\begin{displaymath}
 \D \otimes  \biggl( \bigoplus_{J} \overline{\Omega}_J \otimes V_J \biggr) \twoheadrightarrow 
 E \V_n \rightarrow 0
\end{displaymath}
using the standard filtration $F_\sbt \D$ of $\D$. Recall that $F_k \D$ is the subsheaf of $\D$ of differential operators of order at most $k 
\in \Z$. Set
\begin{displaymath}
\gr_k^F E \V_n \vcentcolon = \xfrac{F_k E \V_n}{F_{k-1} E \V_n} \ \text{ and } \ \gr^F \D \vcentcolon= \textstyle \bigoplus\limits_k 
\displaystyle \gr^F_k \D \cong \Oh{\C^n} [\xi_1, \ldots, \xi_n ] \, .
\end{displaymath}
Let $P \in F_k \D (U)$ for $U \subset \C^n$, open, and $v_I \in V_I$ for $I \in \mathcal{P}(\set{1, \ldots, n})$. We denote by $[P \otimes 
\omega_I \otimes v_I]$ the image of $P \otimes \omega_I \otimes v_I \in F_k \D (U) \otimes  \overline{\Omega}_I \otimes V_I$ in $F_k E\V_n$. 
Furthermore, let $\sigma_k [ P \otimes \omega_I \otimes v_I ]$ be the image of $[ P \otimes \omega_I \otimes v_I ]$ in 
$\gr_k^F E \V_n$. We will prove that $z_i \xi_i$ annihilates $\gr_k^F E \V_n$ for every~$k \in \Z$ and every $i \in \set{1 , \ldots, n}$.  We need to 
distinguish two cases:\\ If $i \notin I$, then 
\begin{align*}
 & z_i \xi_i \cdot \sigma_k [   P \otimes \omega_I \otimes v_I] =  
 \sigma_{k+1} [  P z_i \partial_i  \otimes \omega_I \otimes v_I]  = 
 \sigma_{k+1} [ P \otimes \omega_I \otimes B_{I, I \cup \set{i}} B_{I \cup \set{i}, I} (v_I) ] =0.
\end{align*}
If $i \in I$, then 
\begin{align*}
& z_i \xi_i \cdot \sigma_k [  P \otimes \omega_I \otimes v_I] =  
\sigma_{k+1} [ P \partial_i  z_i\otimes \omega_I \otimes v_I]  =
\sigma_{k+1} [ P \otimes \omega_I \otimes B_{I, I \setminus \set{i}} B_{I \setminus \set{i}, I} (v_I) ]=0 .
\end{align*}
In both cases $z_i \xi_i$ is an annihilator. Thus, the characteristic variety of $E \V_n$ is contained in $\Delta_S$. This also shows us that 
the dimension of the characteristic variety of $E \V_n$ is at most $n = \dim_\C X$ and therefore $E \V_n$ is holonomic. As well, we see that $E \V_n$ is 
a regular holonomic $\D$-module using the fact that $(z_i \xi_i)^1$ is an annihilator of $\gr^F E \V_n = \oplus_k \gr^F_k E \V_n$ (cf.\ \cite[Definition 5.2]{Kashiwara-microlocalcalc}).%
 \end{proof}

 We not that in \cite{KV} a similar but slightly different proof of the holonomicity and the statement on the characteristic variety is given.\\


\subsection[\texorpdfstring{An equivalence with regular singular \Dheading-modules in \Cn{} whose singular locus is a normal crossing
}{An equivalence with regular singular D-modules in C\textasciicircum n whose singular locus is a normal crossing
}]{An equivalence with regular singular \Dheading-modules in \Cn{} whose singular locus is a normal crossing 
}

Let us clarify some notational facts: Let $\iota \colon U \hookrightarrow X$ 
denote the inclusion for an open subset $U$ of $\C^n$. Then $\Gamma_U$ is the functor which maps sheaves on $\C^n$ to sheaves on $\C^n$ defined by
\begin{displaymath}
 \Gamma_U \vcentcolon= \iota_\ast \iota^{-1} .
\end{displaymath}
Moreover, let
\begin{displaymath}
 \C^n = \prod_{i=1}^n \C_i \ \ \text{ and } \ \ W_i \vcentcolon= \C_i \setminus \R^+_0 .
\end{displaymath}
And for $I \in \mathcal{P} (\set{1, \ldots, n})$ set 
\begin{displaymath}
\Lambda_I \vcentcolon= \sum_{k \in I} \Gamma_{\C_k  \times \prod_{\substack{i=1\\ i \neq k}}^n W_i} \Oh{} \ \ \text{ and } \ \ 
\Oh{I} \vcentcolon= \frac{\Gamma_{\prod_{i=1}^n W_i} \Oh{}}{\Lambda_I}.
\end{displaymath}

Note that $(\Gamma_{\prod_{i=1}^n W_i} \Oh{})_0$ and $\Lambda_{I,0}$ are unitary, commutative rings w.\,r.\,t.\ addition and 
multiplication of functions. And $\Oh{I,0}$ is an abelian additive group and a unitary left $\D_0$-module. But $\Oh{I,0}$ is not 
a ring, as in general the multiplication of functions is not well-defined.\\

The following theorem of A.\ Galligo, M.\ Granger and Ph.\ Maisonobe will be important for our computations:

\begin{thm}[\cite{GGM1}, \cite{GGM2}]
The contravariant functor $\mathscr{A}$ from $\mathcal{M}od_{\text{rh}}^S (\D)$ to $\Cc_n$
\begin{align*}
 \mathscr{A} \colon  \mathcal{M}od_{\text{rh}}^S (\D) & \longrightarrow \Cc_n\\
 \M & \longmapsto 
\xymatrixcolsep{2.5pc}\xymatrix{\Hom_{\D_{X,0}}( \M_0, \Oh{I,0}) \ar@<0.7ex>[r]^{\can_{I,i}}  & \Hom_{\D_{X,0}}( \M_0, \Oh{I,0})  \ar@<0.7ex>[l]^{\var_{I,i}}}
\end{align*}
establishes an equivalence of categories. $\can_{I,i}$ is the canonical map or quotient map which sends solutions with values in 
$\Oh{I,0}$ to solutions with values in $\Oh{I \cup \set{i},0}$. $\var_{I,i}$ is the variation around $z_i=0$, i.\,e.\ we have
\begin{displaymath}
\var_{I,i}(F) = M_i F - F \  \text{ for } \  F \in \Hom_{\D_{X,0}}( \M_0, \Oh{I \cup \set{i},0}) 
\end{displaymath}
where $M_i F$ is the class of a representative of $F$ after analytic continuation around the axis $z_i = 0$.\\
A $\D$-linear morphism $\phi \colon \M \rightarrow \mathscr{N}$ in $\mathcal{M}od_{\text{rh}}^S (\D)$ is mapped under $\mathscr{A}$ to the morphism
\begin{displaymath}
 \left( \Hom_{\D_{X,0}} (\phi_0, \Oh{I,0} ) \right)  \text{ in } \, \Cc_n
\end{displaymath}
where $\Hom_{\D_{X,0}} (\phi_0, \Oh{I,0} ) \colon \Hom_{\D_{X,0}} (\mathscr{N}_0, \Oh{I,0} ) \rightarrow \Hom_{\D_{X,0}} (\mathscr{M}_0, 
\Oh{I,0} ) $ 
is given by $g \mapsto g \circ \phi_0$.
\end{thm}

In their paper \cite{GGM1}, Galligo, Granger and Maisonobe prove that the category of perverse sheaves in $\C^n$ with respect to the normal crossing 
stratification $\Xi$ is equivalent to the category of quiver representations $\Cc_n$. They establish a functor
\begin{displaymath}
 \alpha \colon \text{Perv}^\Xi (\C^n) \rightarrow \Cc_n.
\end{displaymath}
Composing the functor $\alpha$ with the solution functor $\text{Sol}$ one receives an equivalence of the categories $\mathcal{M}od_{\text{rh}}^S 
(\D)$ and $\Cc_n$ whereby the functor $\mathscr{A}$ is naturally isomorphic to $\alpha \circ \text{Sol}$ \cite{GGM2}.

Now, we are ready to state and prove the Main theorem:

\begin{thm} \label{thm:main_theorem}
The functors $\mathscr{A} \circ E$ and $Q \circ D$ are naturally isomorphic, i.\,e.\ the following diagram commutes up to a natural isomorphism
\begin{displaymath}
\begin{tikzcd}[column sep=7ex, row sep=6ex]
\mathcal{M}od_{\text{rh}}^S (\D)
\arrow[sloped]{r}[above]{\mathscr{A} \cong \alpha \circ \text{Sol} } & 
\Cc_n \\
\Quis_n
\arrow[]{u}[left]{E}
\arrow[sloped]{r}[above]{D } & 
\Quis_n \ .
\arrow[]{u}[right]{Q}
\end{tikzcd}
\end{displaymath}
In particular, $E \colon \Quis_n \rightarrow \mathcal{M}od_{\text{rh}}^S (\D)$ is an equivalence of categories with quasi-inverse $D \circ 
\mathcal{G} \circ \mathscr{A} $, and $E \circ D \circ \mathcal{G}$ is a quasi-inverse of $\mathscr{A}$.\\
Furthermore, $E$ is essentially surjective. This means the category of quiver $\D$-modules is exactly the category $\mathcal{M}od_{\text{rh}}^S (\D)$, and every $\D$-module in 
$\mathcal{M}od_{\text{rh}}^S (\D)$ is in fact isomorphic to a quiver $\D$-module 
as given in \prettyref{defn:quiver_dmodule}.
\end{thm}

In \cite[Proposition 4.4]{KV} an equivalence of categories of quiver $\D$-modules in the case of a central arrangement of hyperplanes is also stated. But 
the essential image of the equivalence is not completely clarified. As domain they use a full subcategory of the category of 
representations over the quiver corresponding to the arrangement. This subcategory is defined by restricting the eigenvalues of several maps involved 
in the quiver representation. In the case of a normal crossing, this restriction is much more rigid than our restriction from 
$\Qui_n$ to $\Quis_n$. This is a strong evidence that the essential image of their equivalence in our setting is not $\mathcal{M}od^S_\text{h} (\D)$ or 
$\mathcal{M}od^S_{\text{rh}} (\D)$.

The main parts of the proof of \prettyref{thm:main_theorem} are accomplished in \prettyref{prop:thm_proof_prop_1} and \prettyref{prop:thm_proof_prop_2}. 
Before applying the functor $\mathscr{A}$ to our quiver $\D$-modules we state some properties of $\Oh{I,0}$ in \prettyref{lem:zj_dj_bij} and \prettyref{lem:z_r^A} to simplify the arguments later.

\begin{lem} \label{lem:zj_dj_bij}
For $I \in \mathcal{P} ( \set{1, \ldots ,n})$ let $\Oh{I}$ as above. Then,
\begin{itemize}
 \item [(i)] $z_j$ acts bijective on $\Oh{I,0}$ if and only if $j \notin I$.
 \item [(ii)] $\partial_j$ acts bijective on $\Oh{I,0}$ if and only if $j \in I$.
\end{itemize}
\end{lem}

\begin{proof}
Let $j \in \mathcal{P} ( \set{1, \ldots, n})$.
\begin{itemize}
\item [(i)] \begin{spacing}{1.2} For $k = 1, \ldots  , n$ we use $Z_k$ as dummy for $\C_k$ or $W_k$. 
The inverse of $z_j$ fulfils that \mbox{$\frac{1}{z_j} \in (\Gamma_{ \prod_{i =1 }^n Z_i } \Oh{})_0$} if and only if $Z_j = 
W_j$. Thus $z_j$ acts bijective on $(\Gamma_{ \prod_{i =1 }^n Z_i } \Oh{})_0$ if and only if $Z_j = W_j$. As $\Lambda_{I,0}= (\sum_{k \in I} \Gamma_{\C_k  
\times \prod_{i=1, i \neq k}^n W_i} \Oh{})_0$, we immediately see that $z_j$ acts bijective on $\Oh{I,0} \cong \frac{(\Gamma_{\prod_{i=1}^n W_i} 
\Oh{})_0}{\Lambda_{I,0}}$ if and only if $j \notin I$. 
 \item [(ii)]  Let $f(z_1, \ldots , z_n) \in (\Gamma_{\prod_{i=1}^n W_i} \Oh{})_0$. As ${\prod_{i=1}^n W_i}$ is simply connected 
there exists a function  $F(z_1, \ldots, z_n) \in (\Gamma_{\prod_{i=1}^n W_i} \Oh{})_0$ such that $\partial_j F = f$. The other 
primitives of $f$ w.\,r.\,t.\ $\partial_j$ are given by $F(z_1, \ldots, z_n) + C(z_1, \ldots , z_{j-1}, z_{j+1}, \ldots ,z_n)$, where $C  \in  
(\Gamma_{\C_j  \times \prod_{i=1, i  \neq j}^n W_i} \Oh{})_0 $ as $C$ does not depend on $z_j$.
Clearly, $j \in I$ if and only if for any such $C$ it follows that $C \in \Lambda_{I,0}$. Now, we see that functions in $\Oh{I,0}$ have a uniquely defined 
primitive w.\,r.\,t.\ $\partial_j$ if and only if $j \in I$ (constants etc.\ move into the denominator of $\Oh{I,0}$). \qedhere 
\end{spacing}
\end{itemize}%
\end{proof}

\begin{lem} \label{lem:z_r^A}
For every $r \in \set{1, \ldots,n}$ fix a branch of the logarithm on $\C_r \setminus \R_{\geq 0}$. Let $M \in \N^+$ and let $A$ denote a $M \times 
M$-matrix with values in $\C$. We set
\begin{displaymath}
 z_r^A \vcentcolon= \exp ( A \cdot \ln (z_r)).
 \end{displaymath}
$z_r^A$ is considered as a matrix with entries in $(\Gamma_{\prod_{i=1}^n W_i} \Oh{})_0$ and all entries of $z_r^A$ are invertible w.\,r.\,t.\ multiplication of functions in $(\Gamma_{\prod_{i=1}^n 
W_i} \Oh{})_0$ and  $(\Gamma_{\C_{t} 
\times \prod_{i=1, i \neq t}^n W_i} \Oh{})_0$ for $t \neq r$. Then:
\begin{itemize}
 \item [(i)] The matrix $z_r^A$ is invertible in $(\Gamma_{\prod_{i=1}^n W_i} \Oh{})_0$ and $ (\Gamma_{\C_{t} \times 
\prod_{i=1, i \neq t}^n W_i } \Oh{})_0 $ for $t \neq r$.
 \item [(ii)] Let $I = \set{m_1 , \ldots , m_{|I|}} \in \mathcal{P} ( \set{1, \ldots, n})$ and $\set{l_1, \ldots, l_{n-|I|}} = 
\set{1, \ldots ,n} \setminus I$. Assume we are given pairwise commuting $M \times M$-matrices $ A_{m_1}, \ldots, A_{m_{|I|}}, A_{l_1}, \ldots, 
A_{l_{n-|I|}}$ with values in~$\C$, and the eigenvalues of $A_{m_1}, \ldots, A_{m_{|I|}}$ lie in $\Sigma$. 
Let $\lambda = (\lambda_1 , \ldots, \lambda_M)^T \in \C^M$  and 
\begin{gather*}
\tilde{\mathcal{F}} \vcentcolon= 
z_{l_1}^{A_{l_1}} \cdot \ldots \cdot  z_{l_{n-|I|}}^{A_{l_{n-|I|}}} \cdot z_{m_1}^{A_{m_1}} \cdot \ldots \cdot  z_{m_{|I|}}^{A_{m_{|I|}}}.\\
\begin{aligned}
\text{Then:} \hskip 1em &  & \partial_{m_1}^{-1}  \ldots  \partial^{-1}_{m_{|I|}} \tilde{\mathcal{F}} \cdot \lambda  \in (\Lambda_{I,0})^M \  \Longleftrightarrow \ \lambda 
= (0, \ldots, 0)^T  \hskip 1.5em
 \end{aligned}
 \end{gather*}
\end{itemize}
\end{lem}

\begin{proof}
\begin{itemize}
\item [(i)] This becomes clear by passing to the Jordan normal form $J$ of $A$. Let $\mu_1, \ldots , \mu_q \in \C$ denote the eigenvalues of $A$. Then,
\begin{displaymath}
 \det ( \exp (A \cdot \ln (z_r)) ) = \det ( \exp (J \cdot \ln (z_r))) = \prod_{i=1}^q ( z_r^{\mu_i} )^{p_i} \neq 0
\end{displaymath}
where $p_1, \ldots, p_q \in \N^+ $. This yields the invertibility of $\exp (A \cdot \ln (z_r))$.
\item [(ii)] We prove ``$\Rightarrow$''. For simplicity let $I = \set{1, \ldots, |I|}$. By part (i), the claim is equivalent to 
\begin{align*}
 \partial_{1}^{-1}    z_{1}^{A_{1}} \cdot \ldots \cdot   \partial^{-1}_{{|I|}}
z_{{|I|}}^{A_{{|I|}}} \cdot \left( \begin{smallmatrix}
                                \lambda_1 \\ \vdots \\ \lambda_M
                               \end{smallmatrix} \right)
                               \in (\Lambda_{I,0})^M .
\end{align*}
This means, for $l =1, \ldots, |I|$, we find $f_l (z_1, \ldots, z_n)  \in  \Big( (\Gamma_{\C_l  \times \prod_{i=1,  i \neq l}^n W_i} \Oh{})_0 
\Big)^M $ such that 
\begin{align*}
 \partial_{1}^{-1}    z_{1}^{A_{1}} \cdot \ldots \cdot   \partial^{-1}_{{|I|}}
z_{{|I|}}^{A_{{|I|}}} \cdot \left( \begin{smallmatrix}
                                \lambda_1 \\ \vdots \\ \lambda_M
                               \end{smallmatrix} \right) = \sum_{l=1}^{|I|} f_l (z_1, \ldots, z_n) .
\end{align*}
Let us apply $\var_{|I|} \circ \cdots \circ \var_{1}$ to both sides of the equation where $\var_l$ was given by $M_l - \Id$:
\begin{itemize}
 \item [$\bullet$] Let us treat the (LHS): We receive
\begin{displaymath}
 (\var_{|I|} \circ \ldots \circ \, \var_1) ( \text{LHS})= 
 \var_1 ( \partial^{-1}_1 z_1^{A_1} ) \cdot \ldots \cdot \var_{|I|} ( \partial^{-1}_{|I|} z_{|I|}^{A_{|I|}} ) \cdot \left( \begin{smallmatrix}
                                \lambda_1 \\ \vdots \\ \lambda_M
                               \end{smallmatrix} \right) .
\end{displaymath}
Let us prove that $\var_l ( \partial^{-1}_l z_l^{A_l})$ is invertible in $(\Gamma_{\prod_{i=1}^n W_i} 
\Oh{})_0$: We may pass to a single Jordan block $J_a$ with eigenvalue $a \in \Sigma$ for our arguments. 
 First, let $a \neq -1$. Then
\begin{align*}
\var_l \Big( \partial^{-1}_l \exp ( J_{a} \ln (z_l) ) \Big) = 
(J_{a} + \Id)^{-1} \cdot \exp ((J_{a} + \Id) \ln (z_l)) \cdot \big( e^{2\pi i (J_{a} + \Id)} - \Id \big) .
\end{align*}
Using part (i) and $\Spec (J_{a} + \Id) \subset \C \setminus \Z$, this is an invertible matrix in 
$(\Gamma_{\prod_{i=1}^n W_i} \Oh{})_0$.
Now, let $a=-1$.  
The matrix $\partial^{-1}_l \exp ( J_{-1} \ln (z_l))$ is (up to a matrix which is independent of $z_l$) an upper-triangular matrix with $\ln 
(z_l)$ on the diagonal. Hence, the matrix \mbox{$\var_l \big( \partial_l^{-1} \exp ( J_{-1} \ln (z_l) ) \big)$} is an upper-triangular matrix with $2 
\pi i$ as diagonal entry, and therefore it is invertible in~$(\Gamma_{\prod_{i=1}^n W_i} \Oh{})_0$.

\item [$\bullet$] Now, consider the (RHS): 
Using $M_1 f_1 = f_1$ and $M_1 \big( \sum_{l=1}^{|I|} f_l  \big) = M_1 f_1 + M_1 \big( \sum_{l=2}^{|I|} f_l  \big) $, we receive 
\begin{align*}
&(\var_{|I|} \circ \ldots \circ  \var_1) ( \text{RHS})=  
( \var_{|I|} \circ \ldots \circ  \var_2) \Big(  M_1 \big( \textstyle \sum\limits_{l=1}^{|I|} 
\displaystyle f_l  \big) - \textstyle \sum\limits_{l=1}^{|I|} \displaystyle f_l  \Big)
 =\\ &= ( \var_{|I|} \circ \ldots \circ  \var_2) \Big(  M_1 \big( \textstyle \sum\limits_{l=2}^{|I|} \displaystyle f_l  \big) - 
\textstyle \sum\limits_{l=2}^{|I|} \displaystyle  f_l \Big) = (\var_{|I|} \circ \ldots \circ \var_1) ( \textstyle  \sum\limits_{l=2}^{|I|} 
\displaystyle f_l).
\end{align*}
As the variations commute on the left hand side (LHS), we receive furthermore 
\begin{align*}
&(\var_{|I|} \circ \ldots \circ \var_1) \big( \textstyle \sum\limits_{l=2}^{|I|} \displaystyle f_l \big) = 
(\var_{|I|} \circ \ldots \circ \var_3 \circ 
\var_1 \circ \var_2 ) \big( \textstyle \sum\limits_{l=2}^{|I|}  \displaystyle  f_l \big) = \\
& =(\var_{|I|} \circ \ldots \circ \var_3 \circ \var_1 \circ \var_2) \big( \textstyle \sum\limits_{l=3}^{|I|} f_l \big)  .%
\end{align*}
Continuing this process, it yields
$ \displaystyle (\var_{|I|} \circ \ldots \circ  \var_1) ( \text{RHS}) = \left( 0, \ldots, 0 \right)^T$.
\end{itemize}
This leads to the equality
\begin{align*}
\var_1 ( \partial^{-1}_1 z_1^{A_1} ) \cdot \ldots \cdot \var_{|I|} ( \partial^{-1}_{|I|} z_{|I|}^{A_{|I|}} ) \cdot \left( \begin{smallmatrix}
                                \lambda_1 \\ \vdots \\ \lambda_M
                               \end{smallmatrix} \right) = \left( \begin{smallmatrix}
                                0 \\ \vdots \\ 0
                               \end{smallmatrix} \right) . 
\end{align*}
The invertibility of all the matrices $\var_l ( \partial^{-1}_l z_l^{A_l})$ gives  \mbox{$\lambda_1 = \ldots = \lambda_M = 0$} as claimed. \qedhere
\end{itemize}
\end{proof}

\vskip 1em
Now, let us consider how the quiver representation looks like after applying $\mathcal{A}$ to a quiver \mbox{$\D$-module}.

\begin{prop} \label{prop:thm_proof_prop_1}
Let $ \V_n = (V_I, B_{I \cup \set{i}, I}, B_{I, I \cup \set{i}})$ denote an object in $\Quis_n$ with $E \V_n$ as corresponding quiver 
$\D$-module. Then, for every $I \in \mathcal{P}(\set{1, \ldots , n})$, we are given a canonical isomorphism 
\begin{displaymath}
\mathfrak{a} \colon V_I^\ast \overset{\cong }{\longrightarrow} \Hom_{\D_{X,0}}( (E\V_n)_0, \Oh{I,0}) \, .
\end{displaymath}
\end{prop}

\begin{proof}
We abbreviate $\V = \V_n$. The proof of this lemma will be carried out in several steps:
\begin{itemize}
 \item [(i)] We are given the following natural isomorphism:
 \begin{gather*}
 \begin{aligned}
\Hom&_{\D_{0}}( (E \V)_0, \Oh{I,0}) = \biggl\{ 
\phi \in \Hom_{\D_{0}}  \Bigl( \textstyle \bigoplus\limits_{J } \displaystyle
( \D_0 \otimes \overline{\Omega}_J \otimes  V_J ) ,
\Oh{I,0} \Bigr) 
\, \Big| \, \\
 & \phi  ( \partial_j \otimes \omega_J \otimes v_J - 
1 \otimes \omega_{J  \cup \set{j}} \otimes B_{J \cup \set{j}, J}(v_J)) = 0, \\
 &\phi ( z_j \otimes \omega_{J \cup \set{j}} \otimes v_{J \cup \set{j}} - 
1 \otimes \omega_J \otimes B_{J, J \cup \set{j}} (v_{J \cup \set{j}}) ) =0 \text{  for } J \neq \set{1, \ldots, n}, j \notin J \biggr\} \cong \\
 \cong  \biggl\{ & \textstyle \bigoplus\limits_{J} \displaystyle \phi^J \in \textstyle \bigoplus\limits_{J } \displaystyle \Hom_{\C}  \left(   V_J , 
\Oh{I,0} \right) 
\, \Big| \, \partial_j \cdot \phi^J(  v_J )- 
\phi^{J \cup \set{j}} ( B_{J \cup \set{j}, J}(v_J)) = 0,  \\
 & z_j \cdot \phi^{J \cup \set{j}} ( v_{J \cup \set{j}}) - \phi^J (
 B_{J, J \cup \set{j}} (v_{J \cup \set{j}}) ) =0  \text{  for } J \neq \set{1, \ldots, n}, j \notin J \biggr\} .
\end{aligned}%
\end{gather*}

\item [(ii)] Consider the following system of equations from step (i)
\begin{align*}
\begin{split}
 \partial_j \cdot \phi^J -  \phi^{J \cup \set{j}} \circ B_{J \cup \set{j}, J} &= 0 \\
z_j \cdot \phi^{J \cup \set{j}} -  \phi^J \circ B_{J, J \cup \set{j}} &=0 
\end{split}  \tag{$\widetilde{\star}$} \label{eq:PDE_1}
 \end{align*}
where $\phi^J \in \Hom_{\C}  \left(   V_J , \Oh{I,0} \right)$ and $J \neq \set{1, \ldots, n}$, $j \in \set{1, \ldots, n} \setminus J $. 
We use the following algorithm \prettyref{alg:ALG} to express step by step every $\phi^K$ uniquely in 
terms of $\phi^I$ for $K \in \mathcal{P} ( \set{1, \ldots ,n}) \setminus I$: \label{alg:ALG}

 \par
\begingroup
\leftskip=1.1em
We can express $K \cup I$ as the union of the three disjoint sets $K_1 \vcentcolon= K \cap I$, $K_2 \vcentcolon= K \setminus K_1$, $K_3 \vcentcolon= I 
\setminus K_1$. Then $K_2 \neq \varnothing$ or $K_3 \neq \varnothing$, as $K \neq I$. 
\prettyref{lem:zj_dj_bij} yields that $z_l$, $\partial_m$ act bijective on $\Oh{I,0}$ for $l \in K_2$ and $m \in 
K_3$.\\
1.\ Step: If $K_2 = \varnothing$, skip this step. Otherwise we have for $l_1 \in K_2$: 
$$z_{l_1} \cdot  \phi^K - \phi^{K \setminus \set{l_1}} \circ B_{K \setminus \set{l_1}, K}= 0 \ \Longleftrightarrow \ \phi^K = \frac{1}{z_{l_1}} \cdot 
\left( \phi^{K \setminus \set{l_1}} \circ B_{K \setminus \set{l_1}, K} \right)  $$
For $l_2 \in K_2 \setminus \set{l_1}$ use the equation
$$z_{l_2} \cdot \phi^{K \setminus \set{l_1}} - \phi^{K \setminus \set{l_1,l_2}} \circ B_{K \setminus \set{l_1,l_2}, K \setminus \set{l_1}} =0 $$ 
to express $\phi^K$ in terms of $\phi^{K \setminus \set{l_1,l_2}}$. Continue until $\phi^K$ is expressed in terms of $\phi^{K_1}$.\\
2. Step: If $K_3 = \varnothing$, we already expressed $\phi^K$ in terms of $\phi^I$. Otherwise we have for $m_1 \in K_3$: 
\begin{displaymath} \hspace*{1.1em}
\partial_{m_1} \cdot \phi^{K_1}  -   \phi^{K_1 \cup \set{m_1}} \circ B_{K_1 \cup \set{m_1}, K_1}= 0\ \Longleftrightarrow \ \phi^{K_1} =  
\partial_{m_1}^{-1} \cdot \left(  \phi^{K_1 \cup \set{m_1}} \circ B_{K_1 \cup \set{m_1}, K_1} \right) 
\end{displaymath}
For $m_2 \in K_3 \setminus \set{m_1}$ use the equation 
\begin{gather*}
\partial_{m_2}  \cdot \phi^{K_1 \cup \set{m_1}}  -   \phi^{K_1 \cup \set{m_1, m_2}} \circ B_{K_1 \cup \set{m_1,m_2}, K_1 \cup \set{m_1}} = 0
\end{gather*}
to express $\phi^{K_1}$ in 
terms of $\phi^{K_1 \cup \set{m_1, m_2}}$. Continue until $\phi^{K_1}$ -- and therefore $\phi^K$ -- is expressed in terms of~$\phi^I$.\\
The order in which we solve for $\phi^I$ in \prettyref{alg:ALG} does not influence the result. This is ensured by the commutativity conditions on the maps $B_{\sbt,\sbt}$ and the fact that 
$z_l$, $\partial_m$ commute for 
$l \notin I$, $m \in I$. Therefore every $\phi^K$ can be uniquely expressed in terms of~$\phi^I$.
\par%
\endgroup%

So clearly \eqref{eq:PDE_1} implies that $\phi^I \in \Hom_{\C} \left(V_I , \Oh{I,0} \right)$ fulfils the system 
\begin{gather*}
 \begin{aligned}
z_l \partial_l \cdot \phi^I -  \phi^I \circ (B_{I , I \cup \set{l}} B_{I \cup \set{l}, I})&= 0 \\
z_m \partial_m \cdot \phi^I -  \phi^I \circ (B_{I , I \setminus \set{m}} B_{I \setminus \set{m}, I} - \Id)&= 0 \end{aligned}
\tag{$\star$} \label{eq:PDE_2}
\end{gather*}
 of $ n$ equations where $l \notin I$, $m \in I$. On the other hand \eqref{eq:PDE_1} is likewise implied 
by \eqref{eq:PDE_2} using \prettyref{alg:ALG} as definition of $\phi^K$ for all $K \in \mathcal{P}(\set{1, \ldots, n}) \setminus I$. 
This shows us that in fact
\begin{align*}
\Hom&_{\D_{0}}( (E\V)_0, \Oh{I,0}) \cong \biggl\{  \phi^I \in  \Hom_{\C}  \left(   V_I , \Oh{I,0} \right) 
\, \Big| \,   
z_l \partial_l \cdot \phi^I (v_I) - \phi^I (B_{I , I \cup \lbrace l \rbrace } B_{I \cup \set{l}, I} (v_I)) = 0 , \\
\begin{split}
&  z_m \partial_m \cdot \phi^I (v_I) - \phi^I ((B_{I , I \setminus \set{m}} B_{I \setminus \set{m}, I} - \Id) (v_I))   = 0  \text{for } l \in \set{1, \ldots,n} \setminus I, \, m \in I \biggr\}.
\end{split} \tag{1} \label{eq:hom}
\end{align*}

 \item [(iii)] The dimension of $\Hom_{\D_0}((E \V)_0, \Oh{I,0}))$ over $\C$ is finite (see \cite{GGM2}). 
We use the following proposition of \cite{GGM2} to give an upper bound hereof:
 \par
\begingroup
\leftskip=1em
\rightskip=1em
Let $z_I^\ast = (z_1 , \ldots , z_n, \xi_1, \ldots , \xi_n ) \in T^\ast \C^n $ verifying $z_i \xi_i = 0$ for all $i$, and \\ $z_i=0 
\Leftrightarrow i \in I$ and $\xi_i \neq 0 \Leftrightarrow i \in I$. Then $ \dim_\C  \Hom_{\D_{0}}( (E\V)_0, \Oh{I,0}) =  \mult_{z_I^\ast} E \V.$
\par%
\endgroup%
We use the definition of multiplicity as given in \cite[Chapter V]{GM-basiccourse}. Supplementary, we use \cite[Subsection II.B)4]{Serre}.
As the definition becomes clear during the following computations, we do not repeat it here.

We use the good filtration on $E \V$ from the proof of \prettyref{thm:E_functor}. Its sections over $U \subset \C^n$, open, are given by
\begin{displaymath}
 F_k E \V (U) = \frac{F_k \D(U) \otimes  \left( \bigoplus_{J} \overline{\Omega}_J \otimes V_J \right)}{ \left( F_k \D (U) \otimes  ( \bigoplus_{J} 
\overline{\Omega}_J \otimes V_J)  \right) \cap \mathcal{J}(U)}
\end{displaymath}
for $k \in \N_0$, and for $k \in \Z \setminus \N_0$ we have $F_k E \V = 0$. 
As before let
\begin{displaymath}
\gr_k^F E \V = \xfrac{F_k E \V}{F_{k-1} E \V}  \ \text{ and } \  \gr^F E\V = \textstyle \bigoplus\limits_{k \in \N_0} \displaystyle 
\gr_k^F E \V .
\end{displaymath}
Let $k \in \N_0$. Fix a point $\tilde{z}_I^\ast = (\tilde{z}_1, \ldots, \tilde{z}_n, \tilde{\xi}_1 , 
\ldots, \tilde{\xi}_n) =\vcentcolon (\tilde{z}_I, \tilde{\xi}_I)$ where $\tilde{z}_i=0 \Leftrightarrow i \in I$ and  
$\tilde{\xi}_i \neq 0 \Leftrightarrow i \in I$. Consider the stalk of $\gr^F E \V$ at $\tilde{z}_I$. Set
$$M \vcentcolon= (\gr^F E\V)_{\tilde{z}_I}.$$
$z_i$ is an invertible element in $F_k \D_{\tilde{z}_I} $ iff $i \notin I$. So let $i \notin I$ and $K \in \mathcal{P}( \set{1, 
\ldots, n} \setminus \set{i} ) $. We denote by $[ P \otimes 
\omega_{K \cup \set{i}} \otimes v_{K \cup \set{i}}]$ the image of
$P \otimes \omega_{K \cup \set{i}} \otimes v_{K \cup \set{i}} \in F_k \D_{\tilde{z}_I} 
\otimes \overline{\Omega}_{K \cup \set{i}} \otimes V_{K \cup \set{i}}$ 
in $(F_k E\V)_{\tilde{z}_I}$.
We have the following identity in $(F_k E \V)_{\tilde{z}_I}$:
\begin{displaymath}
 [P \otimes \omega_{K \cup \set{i}} \otimes v_{K \cup \set{i}} ] = [ z_i^{-1} P \otimes \omega_K \otimes B_{K, K \cup \set{i}} (v_{K \cup \set{i}}) ]
\end{displaymath}
This allows us to ``eliminate'' all summands $[F_k \D_{\tilde{z}_I} \otimes \overline{\Omega}_J \otimes V_J] $ in $(F_k E \V)_{\tilde{z}_I}$ with $J 
\setminus I \neq \varnothing$. Hence, we may assume that 
\begin{displaymath}
( F_k E \V )_{\tilde{z}_I} = 
\frac{F_k \D_{\tilde{z}_I} \otimes  \bigl( \bigoplus_{J \subseteq I} \overline{\Omega}_J \otimes V_J \bigr)}{ \bigl( F_k 
\D_{\tilde{z}_I} \otimes  ( \bigoplus_{J \subseteq I} \overline{\Omega}_J \otimes V_J)  \bigr) \cap \mathcal{J}_{\tilde{z}_I}} 
\ \text{ or } \ E \V = \frac{\D \otimes  \bigl( \bigoplus_{J \subseteq I} \overline{\Omega}_J \otimes V_J \bigr)}{ \bigl( \D \otimes  ( \bigoplus_{J 
\subseteq I} \overline{\Omega}_J \otimes V_J)  \bigr) \cap \mathcal{J}} \, .
 \end{displaymath}
We have $\gr^F \D_{\tilde{z}_I } \cong \Oh{{\tilde{z}_I}} [\xi_1, \ldots, \xi_n]$ and $M$ is a finitely generated $\Oh{{\tilde{z}_I}} [\xi_1, \ldots, \xi_n]$-module. We denote by $\mathcal{M}ax$ the 
maximal ideal of the local ring $\Oh{{\tilde{z}_I}}$. Let 
\begin{displaymath}
Q_{\tilde{\xi}_I} \vcentcolon= \mathcal{M}ax + ( \xi_1 - \tilde{\xi}_1, \ldots , \xi_n - \tilde{\xi}_n) . 
\end{displaymath}
This defines a maximal ideal in $\Oh{{\tilde{z}_I}} [\xi_1, \ldots, \xi_n]$.  
Thus, $\xfrac{M}{Q_{\tilde{\xi}_I} M}$ is a finitely generated $\xfrac{\Oh{{\tilde{z}_I}} [\xi_1, \ldots, \xi_n]}{Q_{\tilde{\xi}_I}} $-vector space. 
Therefore, there exists a polynomial $P_{M, Q_{\tilde{\xi}_I}} (N)$, called Hilbert-Samuel polynomial, and an integer $N_0 \in \N$ such that 
\begin{displaymath}
P_{M, Q_{\tilde{\xi}_I}} (N) = \length (\xfrac{M}{Q_{\tilde{\xi}_I}^{N} M})  \ \text{ for all } N \geq N_0.
\end{displaymath}
The highest degree term of $P$ has the form $\displaystyle \frac{e}{d!} \, N^d$ where $ e \in \N, \, d \in \N$ and by definition
$$e = \mult_{\tilde{z}_I^\ast} E \V.$$
Applying \cite[Proposition 11a) in Subsection II.B)4]{Serre}, we receive
\begin{displaymath}
 P_{M, Q_{\tilde{\xi}_I}} (N) =  
 P_{T^{-1} M, T^{-1} Q_{\tilde{\xi}_I} } (N) \ \text{ for } \ T \vcentcolon= \Oh{{\tilde{z}_I}} [\xi_1, \ldots, \xi_n] 
\setminus Q_{\tilde{\xi}_I}.
\end{displaymath}
So we need to consider the localisation of $M$ at $T$:
\begin{displaymath}
T^{-1} M = \textstyle \bigoplus\limits_{k \in 
\N_0} \displaystyle T^{-1} (\gr_k^F E \V)_{\tilde{z}_I} 
\end{displaymath}
Let $[P \otimes \omega_K \otimes v_K]$ denote the image of $P \otimes \omega_K \otimes v_K \in F_k \D_{\tilde{z}_I} \otimes \overline{\Omega}_K \otimes 
V_K $ in $(\gr^F_k E\V)_{\tilde{z}_I}$ for $K \subsetneqq I$. For every $i \in I \setminus K$ we have the following identity in $(\gr^F_{k+1} 
E\V)_{\tilde{z}_I}$:
\begin{displaymath}
\xi_i \cdot [ P \otimes \omega_K \otimes v_K ] =  [ P \otimes \omega_{K \cup \set{i}} \otimes B_{K \cup \set{i}, K} (v_K)] = 0
\end{displaymath}
Consider this identity in $T^{-1} M$: The map $\xi_i \cdot \_ \colon T^{-1} (\gr_k^F E \V)_{\tilde{z}_I} \rightarrow T^{-1} (\gr_{k+1}^F E 
\V)_{\tilde{z}_I}$ is bijective for $i \in I$, as $\tilde{\xi}_i \neq 0$ for $i \in I$. 
Therefore, $\frac{[ P \otimes \omega_K \otimes v_K ]}{1} =0$ and we may assume that 
\begin{displaymath}
T^{-1} ( F_k E \V )_{\tilde{z}_I} = 
\frac{T^{-1} F_k \D_{\tilde{z}_I} \otimes   \overline{\Omega}_I \otimes V_I }{ \left( T^{-1} F_k 
\D_{\tilde{z}_I} \otimes   \overline{\Omega}_I \otimes V_I  \right) \cap T^{-1} \mathcal{J}_{\tilde{z}_I} } .
\end{displaymath}
Using \cite[Proposition 11a)]{Serre} the other way round, we may assume that 
\begin{displaymath}
( F_k E \V )_{\tilde{z}_I} = 
\frac{F_k \D_{\tilde{z}_I} \otimes   \overline{\Omega}_I \otimes V_I }{ \left( F_k 
\D_{\tilde{z}_I} \otimes   \overline{\Omega}_I \otimes V_I \right) \cap \mathcal{J}_{\tilde{z}_I} } 
\ \text{ or } \
 E \V = \frac{\D \otimes \overline{\Omega}_I \otimes V_I }{\left( \D \otimes \overline{\Omega}_I \otimes V_I \right) \cap \mathcal{J}}.
\end{displaymath}
For simplicity let $I = \set{1, \ldots, |I|}$ for the moment. Set $n_I \vcentcolon = \dim_\C (V_I)$. Consider the following exact sequence of holonomic $\D$-modules
\begin{gather*}
 0 \longrightarrow \ker (\pi) \lhook\joinrel\longrightarrow \widetilde{\mathscr{N}} \overset{\pi}{\relbar\joinrel\twoheadrightarrow} E \V \longrightarrow 
0 \\
\begin{aligned}
\text{where } \widetilde{\mathscr{N}} & \vcentcolon = \xfrac{\D \otimes  V_I }{\left(z_{1} \otimes V_I, \ldots ,z_{{|I|}} \otimes V_I, \partial_{|I|+1}  \otimes V_I, 
\ldots, \partial_{{n} }  \otimes V_I \right)}  \cong \\ &\cong \textstyle \bigoplus\limits_{n_I\text{-times}} \displaystyle  \ \xfrac{\D}{(z_{1}, 
\ldots ,z_{{|I|}}, \partial_{|I|+1}, \ldots, \partial_{ {n} })}  = \vcentcolon \textstyle \bigoplus\limits_{n_I\text{-times}} 
\displaystyle  \ \mathscr{N} .
\end{aligned}
\end{gather*}
This sequence yields  $ \mult_{\tilde{z}_I^\ast} E \V \leq \mult_{\tilde{z}_I^\ast} 
\widetilde{\mathscr{N}}$. 
Furthermore, $ \mult_{\tilde{z}_I^\ast} \widetilde{\mathscr{N}} = n_I \cdot \mult_{\tilde{z}_I^\ast} \mathscr{N}$.
So let us compute $\mult_{\tilde{z}_I^\ast} \mathscr{N}$ where $\tilde{z}_I^\ast = ( 0, \ldots, 0, \tilde{z}_{|I|+1} , \ldots, \tilde{z}_n, \tilde{\xi}_1 , 
\ldots, \tilde{\xi}_{|I|}, 0, \ldots, 0) = \vcentcolon(\tilde{z}_I, \tilde{\xi}_I) $
with $\tilde{z}_{|I|+1} , \ldots, \tilde{z}_n, \tilde{\xi}_1 , 
\ldots, \tilde{\xi}_{|I|} \neq 0$: \\ 
{\setstretch{1.2}We use the good filtration $F_\sbt \mathscr{N}$ on $\mathscr{N}$ which  is induced by the standard filtration 
$F_\sbt \D$ of~$\D$. So, we consider $ (\gr^F \mathscr{N})_{\tilde{z}_I} \cong \C \{ z_{|I|+1} - \tilde{z}_{|I|+1} , \ldots, z_n - 
\tilde{z}_n \} [\xi_1, \ldots,  {\xi}_{|I|} ]$ as a module over $(\gr^F \D )_{\tilde{z}_I} \cong  \C \{ z_1, \ldots 
,z_{|I|}, z_{|I|+1} - \tilde{z}_{|I|+1} , \ldots, z_n - \tilde{z}_n \} [\xi_1, \ldots,  {\xi}_{n}]$. Let $\mathcal{M}ax$ be the maximal ideal of $\C \{ 
z_1, \ldots ,z_{|I|}, z_{|I|+1} - \tilde{z}_{|I|+1} , \ldots, z_n - \tilde{z}_n \}$. We need to compute 
the multiplicity of $(\gr^F \mathscr{N})_{\tilde{z}_I}$ with respect to the maximal ideal $\mathcal{M}ax + ( \xi_1 - \tilde{\xi}_1, \ldots , \xi_{|I|} - 
\tilde{\xi}_{|I|}, \xi_{|I|+1}, \ldots , \xi_n)$ of $(\gr^F \D )_{\tilde{z}_I}$. A shift of coordinates gives us that we equivalently have to treat \par } \vspace*{-1.2\parskip}
\begin{displaymath}
\C \{ z_{|I|+1}  , \ldots, z_n  \} [\xi_1, \ldots,  {\xi}_{|I|} ] \ \text{ as a module over } \ 
  \C \{ z_1, \ldots, z_n  \} [\xi_1, \ldots,  {\xi}_{n} ] ,
     \end{displaymath}
and compute its multiplicity with respect to the maximal ideal
\begin{displaymath}
Q\vcentcolon= (z_1, \ldots, z_n, \xi_1 - \tilde{\xi}_1, \ldots , \xi_{|I|} - 
\tilde{\xi}_{|I|}, \xi_{|I|+1}, \ldots , \xi_n).
\end{displaymath}
So we have to compute
\begin{align*}
\length \left( \frac{\C \{ z_{|I|+1}  , \ldots, z_n  \} [\xi_1, \ldots,  {\xi}_{|I|} ]}{(z_{|I|+1}, \ldots, z_n, \xi_1 - 
\tilde{\xi}_1, \ldots , \xi_{|I|} - 
\tilde{\xi}_{|I|} )^{N} \cdot  \C \{ z_{|I|+1}  , \ldots, 
z_n  \} [\xi_1, \ldots,  
{\xi}_{|I|} ]} \right) .
       \end{align*}
But this is the number of monomials of degree less than $N$ in  $\C \{ z_{|I|+1}  , \ldots, z_n  \} [\xi_1, \ldots,  
{\xi}_{|I|} ]$ which is equal to $\binom{N-1+n}{N-1} $.
This shows us that $\mult_{\tilde{z}_I^\ast} \mathscr{N} =1$ and $\mult_{\tilde{z}_I^\ast} E \V \leq n_I$.

\item [(iv)]  Now, we construct the canonical isomorphism~$\eta_I$ from~$V_I^\ast$ into~\eqref{eq:hom}. For this purpose let \mbox{$\alpha \in V_I^\ast$}. 
We define~$\eta_I ( \alpha)$ as follows:

Let $\set{m_1, \ldots, m_{|I|}} = I$, $\set{l_1, \ldots, l_{n-|I|} } = \set{1 , \ldots, n} \setminus I$. For a moment fix a basis of $V_I$ and 
denote it by $v_{I,1} , \ldots , v_{I,n_I}$. In abuse of notation we denote the matrices corresponding to the maps
$\alpha$, $\mathcal{B}_{I , I \cup \lbrace l \rbrace } \vcentcolon = B_{I , I \cup \lbrace l \rbrace } B_{I \cup \set{l}, I}$ and $ \mathcal{B}_{I , I 
\setminus \set{m}} \vcentcolon = B_{I , I \setminus \set{m}} B_{I \setminus \set{m}, I}$ w.\,r.\,t.\ this basis by the same symbols. 
 We set 
\begin{align*}
 \mathscr{F} & \vcentcolon= z_{l_1}^{\mathcal{B}_{I, I \cup \set{l_1}}} \cdot \ldots \cdot 
 z_{l_{n-|I|}}^{\mathcal{B}_{I, I \cup \set{l_{n-|I|}}}} \cdot
 z_{m_1}^{\mathcal{B}_{I, I \setminus \set{m_1}}-\Id} \cdot \ldots \cdot 
 z_{m_{|I|}}^{\mathcal{B}_{I, I \setminus \set{m_{|I|}}}-\Id} \ \text{ and} \\
 \eta_I (\alpha ) & \vcentcolon = \alpha \cdot \mathscr{F} .
\end{align*}
One verifies directly that $\eta_I (\alpha)$ is indeed an element in \eqref{eq:hom} by plugging it into~\eqref{eq:PDE_2}. 

We need to verify that this construction of $\eta_I$ is independent of the choice of basis of $V_I$. So, let $ \tilde{v}_{I,1}, \ldots, \tilde{v}_{I, 
n_I} $ denote another basis of $V_I$. Let $\tilde{\mathcal{B}}_{I , I \cup \lbrace l \rbrace }$, $\tilde{\mathcal{B}}_{I , I \setminus \set{m}}$ 
and $\tilde{\alpha}$ denote the ma\-tri\-ces corresponding to the linear maps $\mathcal{B}_{I , I \cup \lbrace l \rbrace }$, $\mathcal{B}_{I , I \setminus 
\set{m}}$ and $\alpha$ w.\,r.\,t.\ this new basis. Let~$R$ denote the matrix of the change of coordinates from  $\{ v_{I,1} , \ldots , v_{I,n_I} \}$ to $ 
\{\tilde{v}_{I,1}, \ldots, \tilde{v}_{I, n_I} \}$. 
Let $v_I \in V_I$. We denote by $v_I$ in abuse of notation the vector w.\,r.\,t.\ the basis $\{ v_{I,1} , \ldots , v_{I,n_I} 
\}$ and by $\tilde{v}_I$ the vector w.\,r.\,t.\ the basis $ \{\tilde{v}_{I,1}, \ldots, \tilde{v}_{I, n_I} \}$.  
We receive
\begin{align*}
&\eta_I (\alpha) (v_I) = \alpha \cdot  z_{l_1}^{\mathcal{B}_{I, I \cup \set{l_1}}} \cdot \ldots \cdot 
 z_{l_{n-|I|}}^{\mathcal{B}_{I, I \cup \set{l_{n-|I|}}}} \cdot
 z_{m_1}^{\mathcal{B}_{I, I \setminus \set{m_1}}-\Id} \cdot \ldots \cdot 
 z_{m_{|I|}}^{\mathcal{B}_{I, I \setminus \set{m_{|I|}}}-\Id}  \cdot v_I = \\ 
&=  \tilde{\alpha}  R R^{-1} z_{l_1}^{\tilde{\mathcal{B}}_{I, I \cup \set{l_1}}} R \ldots R^{-1}
 z_{l_{n-|I|}}^{\tilde{\mathcal{B}}_{I, I \cup \set{l_{n-|I|}}}} 
 z_{m_1}^{\tilde{\mathcal{B}}_{I, I \setminus \set{m_1}}-\Id}  R \ldots R^{-1}  
 z_{m_{|I|}}^{\tilde{\mathcal{B}}_{I, I \setminus \set{m_{|I|}}}-\Id}  R  R^{-1} \tilde{v}_I = \\ &= \eta_I (\tilde{\alpha}) (\tilde{v}_I)  .
\end{align*}
Hence, our construction is independent of the choice of basis of $V_I$.

Now, we want to check that $\eta_I$ is injective. So assume that $\eta_I (\alpha)$ is the zero mapping. 
As $\partial_m$ acts bijective on $\Oh{I,0}$ for $m \in I$ (see \prettyref{lem:zj_dj_bij}), this is equivalent to
\begin{displaymath}
 \partial_{m_1}^{-1}  \ldots  \partial_{m_{|I|}}^{-1}  \, z_{l_1}^{\mathcal{B}^T_{I, I \cup \set{l_1}}} \cdot \ldots \cdot 
 z_{l_{n-|I|}}^{\mathcal{B}^T_{I, I \cup \set{l_{n-|I|}}}} \cdot
 z_{m_1}^{\mathcal{B}^T_{I, I \setminus \set{m_1}}-\Id} \cdot \ldots \cdot 
 z_{m_{|I|}}^{\mathcal{B}^T_{I, I \setminus \set{m_{|I|}}}-\Id} \cdot \alpha^T = \left( \begin{smallmatrix}
                                0 \\ \vdots \\ 0
                               \end{smallmatrix} \right) .
\end{displaymath}
The eigenvalues of $\mathcal{B}^T_{I, I \setminus \set{m}}-\Id$ are contained in $\Sigma$ for $m \in I$, as $\V$ is an object in $\Quis_n$. Using \prettyref{lem:z_r^A}, we receive $\alpha \equiv 
0$ and $\eta_I$ is injective.

As $ \dim_\C  \Hom_{\D_{0}}( (E\V)_0, \Oh{I,0}) \leq n_I$ by part (iv), we immediately receive the bijectivity of~$\eta_I$ as claimed.

\end{itemize}

Composing the isomorphism from part (i) with \prettyref{alg:ALG}, we receive a natural isomorphism from \eqref{eq:hom} into $\Hom_{\D_{0}}( (E\V)_0, 
\Oh{I,0})$. Composing this with the  isomorphism $\eta_I$ from $V_I^\ast$ into \eqref{eq:hom}, this gives us the canonical isomorphism
\begin{displaymath}
\mathfrak{a} \colon V_I^\ast \overset{\cong}{\longrightarrow} \Hom_{\D_{0}}( (E\V)_0, \Oh{I,0}) . \qedhere
\end{displaymath}
\end{proof}

\vskip 1em
The following statement on the matrix polynomial will be used in the proof of \prettyref{prop:thm_proof_prop_2}.
\pagebreak

\begin{cor} \label{cor:phi_A_psi_A}
Let $A$ denote a square matrix with entries in $\C$ and let $i \in \set{1, \ldots ,n}$. We fix a branch of the logarithm defined on $\C_i \setminus 
\R^+_0$ and let $z_i^A = \exp ( A \cdot \ln (z_i))$ as before. Set
\begin{displaymath}
 \varphi_A (z_i) \vcentcolon= \sum_{k=0}^\infty \frac{A^k}{(k+1)!}  \cdot
{\ln (z_i)^{k+1}} \ \ \text{ and } \ \ \psi(A)\vcentcolon= \sum_{k=1}^\infty \frac{(2\pi i)^k}{k!} A^{k-1}.
\end{displaymath}
Then  
\begin{displaymath}
M_i \varphi_A ( z_i) - \varphi_A (z_i) =  \psi (A) \cdot z_i^A .
\end{displaymath}
\end{cor}

\begin{proof}
We have
\begin{displaymath}
 M_i \varphi_A (z_i) = \sum_{k=0}^\infty \frac{A^k}{(k+1)!}  \cdot
{(\ln (z_i)+2\pi i)^{k+1}} .
\end{displaymath}
Furthermore direct computation yields
\begin{align*}
A \cdot \left( M_i \varphi_A ( z_i) - \varphi_A (z_i) \right ) = \left(  M_i \varphi_A ( 
 z_i) - \varphi_A (z_i) \right) \cdot A=\psi (A) \cdot  z_i^A \cdot  A =   A \cdot \psi(A) \cdot z_i^A  .
\end{align*}
We may assume for our arguments that $A = J_a$ where $J_a$ is a single Jordan block with eigenvalue $a \in \C$. 
If $a \neq 0$, our claim follows directly. So assume $a=0$. The above equation shows us however that $M_i \varphi_A (  z_i) - \varphi_A (z_i)$ and 
$\psi (A) \cdot z_i^A$ coincide up to a possible difference in the entry in the upper-left corner. The entry of $M_i \varphi_A (  z_i) - \varphi_A (z_i)$ 
in the upper-left corner is~$\ln(z_i) + 2 \pi i - \ln(z_i) = 2 \pi i$. The first column of $z_i^A$ is $(1 , 0 , \ldots, 0)^T$ and the entry in the 
upper-left corner of $\psi (A)$ is $2 \pi i$. Hence, the entry in the upper-left corner of $ \psi (A) \cdot z_i^A$ is~$2 \pi i$ as well 
which proves the claim. \qedhere
\end{proof}

In the following we prove that the quiver representation one receives after applying~$\mathscr{A}$ to a quiver $\D$-module 
is determined in a simple manner by the starting quiver representation. To do so, we ``extend'' the canonical isomorphism $\mathfrak{a}$ from 
\prettyref{prop:thm_proof_prop_1} to the whole quiver~representation.

\begin{prop} \label{prop:thm_proof_prop_2}
Let $ \V_n = (V_I, B_{I \cup \set{i}, I}, B_{I, I \cup \set{i}})$ be an object in $\Quis_n$ and $E \V_n$~the corresponding quiver 
$\D$-module. The image of $E \V_n$ under the functor $\mathscr{A}$ is canonically isomorphic~to
\begin{displaymath}
\begin{tikzcd}[column sep=normal]
V_I^\ast  
\arrow[yshift=0.6ex, sloped]{r}[above]{u_{I,i}}  & 
V_{I \cup \set{i}}^\ast 
\arrow[yshift=-0.6ex, sloped]{l}[below, inner sep=0.8ex]{w_{I,i} }
\end{tikzcd} 
\end{displaymath}
where 
\begin{displaymath}
 u_{I,i} = B_{I,I \cup \set{i}}^\ast  \ \  \text{and} \ \ w_{I,i} = B_{I \cup 
\set{i},I}^\ast  \circ \sum_{k=1}^\infty \frac{(2 \pi i)^k}{k!} (B_{I,I \cup \set{i}}^\ast \circ B_{I \cup 
\set{i},I}^\ast )^{k-1}  .
\end{displaymath}
\end{prop}

\begin{proof}
Let $n_I = \dim_\C V_I$ as before. Set $\mathcal{B}_{K,L} \vcentcolon = B_{K,L} \circ B_{L,K}$ if $ K, L \in \mathcal{P}(\set{1, \ldots, n})$ are adjacent, 
i.\,e.\ $K=L \cup \set{l}$ or $L= K \cup \set{k}$.  The image of $(V_I, B_{I \cup \set{i}, I}, B_{I, I \cup \set{i}})$ under 
$\mathcal{A} \circ E$ is given by  $$(\Hom_{\D_{0}}( (E \V_n)_0, \Oh{I,0}) , \can_{I,i}, \var_{I,i}).$$ 
First, we reperform the first steps of the proof of \prettyref{prop:thm_proof_prop_1}. Then we compute $\can$ and~$\var$.
\begin{itemize}
 \item [(i)] First, note that the natural isomorphism we gave for $\Hom_{\D_{0}}( (E \V_n)_0, \Oh{I,0})$ in part~(i) of the proof of 
\prettyref{prop:thm_proof_prop_1} is compatible with the canonical map and the variation. 
Therefore, 
it extends to the entire object $(\Hom_{\D_{0}}( (E \V_n)_0, \Oh{I,0}) , \can_{I,i}, \var_{I,i})$. We receive:
\begin{gather*}
  \xymatrix@1@C=3em{ \Hom_{\D_{0}}( (E \V_n)_0, \Oh{I,0}) 
\ar@<0.7ex>[r]^-{\can_{I,i}}  & \Hom_{\D_{0}}( (E \V_n)_0, \Oh{I \cup \set{i},0}) \ar@<0.7ex>[l]^-{\var_{I,i}}}  \cong \\
\begin{aligned}
  \biggl\{ & \textstyle \bigoplus\limits_{J} \displaystyle \phi^J_I \in \textstyle \bigoplus\limits_{J} \displaystyle \Hom_{\C}  \left(   V_J , 
\Oh{I,0} \right) 
\, \Big| \, \partial_j \cdot \phi_I^J(  v_J )- 
\phi_I^{J \cup \set{j}} ( B_{J \cup \set{j}, J}(v_J)) = 0, \\
&  z_j \cdot \phi_I^{J \cup \set{j}} ( v_{J \cup \set{j}}) - \phi_I^J (
 B_{J, J \cup \set{j}} (v_{J \cup \set{j}}) ) =0  \text{ for } J \neq \set{1, \ldots, n}, \, j \in \set{1, \ldots, n} \setminus J \biggr\} \end{aligned} \\
\begin{aligned}
   \mathop{\oplus}_J {\can^J_{I,i}} \ \raisebox{1.8ex}{\text{\rotatebox{-90}{$\longrightarrow$}}} \ \  \ 
\raisebox{-1.8ex}{\text{\rotatebox{90}{$\longrightarrow$}}} \ \mathop{\oplus}_J \var^J_{I,i} \end{aligned} \\
\begin{aligned}
   \biggl\{& \textstyle \bigoplus\limits_{J} \displaystyle \phi^J_{I \cup \set{i}} \in \textstyle \bigoplus\limits_{J} \displaystyle \Hom_{\C}  
\left(   V_J , \Oh{{I \cup \set{i}},0} \right) 
\, \Big| \,  \partial_j \cdot \phi_{I \cup \set{i}}^J(  v_J )- 
\phi_{I \cup \set{i}}^{J \cup \set{j}} ( B_{J \cup \set{j}, J}(v_J)) = 0,  \\
&    z_j \cdot \phi_{I \cup \set{i}}^{J \cup \set{j}} ( v_{J \cup \set{j}}) - \phi_{I \cup \set{i}}^J (
 B_{J, J \cup \set{j}} (v_{J \cup \set{j}}) ) =0  \text{ for } J \neq \set{1, \ldots, n}, \, j \in \set{1, \ldots, n} \setminus J \biggr\} 
\end{aligned}
\end{gather*}
\item [(ii)]Let us fix~$I \in \mathcal{P} ( \set{1, \ldots, n}) \setminus 
\set{1, \ldots, n}$, $i \in \set{1, \ldots, n} \setminus I$ temporarily. 
We may consider only the behaviour under the canonical map and the variation of the two pairs
\begin{displaymath}
 \begin{pmatrix}
\phi_I^I \\[1ex] \phi_I^{I \cup \set{i}} 
 \end{pmatrix}
\begin{matrix} \leftrightarrow  \\[1ex] \leftrightarrow \end{matrix}
\begin{pmatrix}
                 \phi_{I \cup \set{i}}^I \\[1ex] \phi_{I \cup \set{i}}^{I \cup \set{i}}
                \end{pmatrix} .
\end{displaymath}%
For any $K \in \mathcal{P} ( \set{1, 
\ldots ,n})$, $\phi_I^K \leftrightarrow \phi_{I \cup \set{i}}^K$ will follow their behaviour under these two maps. This can be seen by 
adapting \prettyref{alg:ALG} in  the following way: 
We only use equations of~\eqref{eq:PDE_1} which 
involve $z_j$ for $j \notin I \cup \set{i}$, or $\partial_k$ for $k \in I$. That way we express $\phi_I^K$ in terms of $\phi_I^I$ if $i \notin K$, 
or in terms of $\phi_I^{I \cup \set{i}}$ if $i \in K$.  
This expression is unique with the same arguments as for \prettyref{alg:ALG}. One observes that $\phi_I^K$, $\phi_{I \cup \set{i}}^K$ are build up from 
$\phi_I^I$, $\phi_{I \cup \set{i}}^I$ if $i \notin K$ (from $\phi_I^{I \cup \set{i}}$, $\phi_{I \cup \set{i}}^{I \cup \set{i}}$ if $i \in K$) in a 
completely identical manner. This ensures that they behave in the same way under the canonical map and the variation.
 \item [(iii)] Using the isomorphisms~$\eta_I$ and $\eta_{I \cup \set{i}}$ from part (iv) of the proof of \prettyref{prop:thm_proof_prop_1}, we can uniquely identify $\phi^I_I$ and 
$\phi^{I \cup \set{i}}_{I \cup \set{i}}$ with elements of $V_I^\ast$ and $V_{I \cup \set{i}}^\ast$, respectively. After a choice of basis of $V_I$ and 
$V_{I \cup \set{i}}$, we may write for some $\alpha_I \in V_I^\ast$ and $\alpha_{I \cup \set{i}} \in V_{I \cup \set{i}}^\ast$ (we omit set braces for 
singletons in the following)
\begin{gather*}
 \begin{aligned}%
\phi_I^I &= \eta_I ( \alpha_I)= \alpha_I \cdot \mathscr{F}_I   &   \phi_{I \cup i}^I &= \partial_i^{-1} \cdot  \phi^{I \cup i }_{I 
\cup i} \cdot B_{I \cup i, I}  \\
 \phi_I^{I \cup i} &= z_i^{-1} \cdot  \phi_I^I \cdot B_{I, I \cup i}  
 & \phi^{I \cup i}_{I \cup i} &= \eta_{I \cup i} (\alpha_{I \cup i}) = \alpha_{I \cup i} \cdot \mathscr{F}_{I \cup i}
 \end{aligned}\\
\begin{aligned}
 \mathscr{F}_I &= z_{i}^{\mathcal{B}_{I, I \cup i}} \cdot z_{l_{2}}^{\mathcal{B}_{I, I \cup l_{2}}} \cdot \ldots \cdot 
 z_{l_{n-|I|}}^{\mathcal{B}_{I, I \cup l_{n-|I|}}} \cdot
 z_{m_1}^{\mathcal{B}_{I, I \setminus m_1}-\Id} \cdot \ldots \cdot 
 z_{m_{|I|}}^{\mathcal{B}_{I, I \setminus m_{|I|}}-\Id}\\
 \mathscr{F}_{I \cup i} &= z_{l_2}^{\mathcal{B}_{I \cup i, I \cup \set{i, l_2}}} \cdot \ldots \cdot 
 z_{l_{n-|I|}}^{\mathcal{B}_{I \cup i, I \cup \set{i ,l_{n-|I|}}}} \cdot 
 z_{i}^{\mathcal{B}_{ I \cup i, I}- \Id} \cdot
 z_{m_1}^{\mathcal{B}_{I \cup i, \set{ I \cup i} \setminus m_1}-\Id} \cdot \ldots \cdot 
 z_{m_{|I|}}^{\mathcal{B}_{I \cup i, \set{I \cup i} \setminus m_{|I|}}-\Id} 
\end{aligned}
\end{gather*}
where $\set{i,l_2,  \ldots, l_{n-|I|} } = \set{1 , \ldots, n} \setminus I$, $\set{m_1, \ldots, m_{|I|}} = I$. 
This description of~$\phi^\sbt_\sbt$ is independent of the choice of basis as we showed in part (iv) of the proof 
of \prettyref{prop:thm_proof_prop_1}.

 Let us give some helpful identities for the computations.  
We have for $i,l \notin I$, $l \neq i$, 
$m \in I$:
\begin{align*}
 B_{I , I \cup i} \cdot \mathcal{B}_{I \cup i, \set{I \cup i} \setminus m} &= \mathcal{B}_{I, I \setminus m} \cdot B_{I , I \cup i} \ & 
 B_{ I \cup i, I} \cdot \mathcal{B}_{I , I \cup l} &=  \mathcal{B}_{I \cup i, I \cup \set{i,l}} \cdot  B_{ I \cup i, I} \\
  B_{I, I \cup i} \cdot  
\mathcal{B}_{I \cup i, I \cup \set{i,l}} &= \mathcal{B}_{I, I \cup l} \cdot  B_{I, I \cup i}   &    
 B_{I \cup i, I} \cdot \mathcal{B}_{I, I \setminus m} &= \mathcal{B}_{I \cup i, \set{I \cup i} \setminus m} \cdot B_{I \cup i, I} 
\end{align*}
 \item [(iv)] We claimed that the canonical map from $ \begin{pmatrix}
\phi_I^I \\[1ex] \phi_I^{I \cup i} 
 \end{pmatrix}$ to $\begin{pmatrix}
                 \phi_{I \cup i}^I \\[1ex] \phi_{I \cup i}^{I \cup i}
                \end{pmatrix}$ is given by $B_{I,I \cup i}^\ast$. This means we have to check that the assignment
 \begin{displaymath}
  \alpha_I \mapsto \alpha_{I \cup i} \vcentcolon = \alpha_I \cdot B_{I , I \cup i} 
 \end{displaymath}
describes the canonical map. This follows by direct computations:
\begin{align*}
&\eta_{I \cup i} ( \alpha_I \cdot B_{I , I \cup i}) =  \alpha_I \cdot B_{I , I \cup i} \cdot \mathscr{F}_{I \cup i}  = 
  z_i^{-1} \cdot \alpha_I \cdot \mathscr{F}_I \cdot B_{I,I \cup i} = z_i^{-1} \cdot \phi_I^I \cdot B_{I,I \cup i} = \phi_I^{I \cup i}
\end{align*}
and therefore
\begin{align*}
 \partial_i^{-1} & \cdot \eta_{I \cup i} ( \alpha_I \cdot B_{I , I \cup i})  \cdot B_{I \cup i, I} = 
\partial_i^{-1} z_i^{-1} \cdot \alpha_I \cdot \mathscr{F}_I \cdot B_{I,I \cup i}  \cdot B_{I \cup i, I} = 
\alpha_I \cdot  \mathscr{F}_I   = \phi_I^I 
\end{align*}
With the same arguments as before, one can show that the 
description of the canonical map is independent of the choice of basis.
 \item [(v)] We are left with the computation of the variation 
$M_i  \phi_{I \cup i}^I - \phi_{I \cup i}^I$ and $ M_i \phi_{I \cup i}^{I \cup i} - \phi_{I \cup i}^{I \cup i}$. 
In particular, we need to check that the assignment
\begin{displaymath}
 \alpha_{I \cup i} \mapsto \alpha_I \vcentcolon = \alpha_{I \cup i} \cdot \biggl(\sum_{k=1}^\infty \frac{(2 \pi i)^k}{k!} ( B_{I \cup 
i,I}  B_{I,I \cup \set{i}})^{k-1} \! \biggr)  \cdot B_{I \cup 
i,I}   = \alpha_{I \cup i} \, \cdot \, \underbrace{ \psi(\mathcal{B}_{I \cup 
i,I}  ) \! \cdot \! B_{I \cup 
i,I} }_{\mathclap{ =\vcentcolon \, \Theta_{I,i}}}  
\end{displaymath}
describes the variation. For $\phi_{I \cup i}^{I \cup i}$ the correctness follows by direct computation:
\begin{align*} 
&M_i \phi_{I \cup i}^{I \cup i} -  \phi_{I \cup i}^{I \cup i} = 
  \alpha_{I \cup i} \cdot (e^{2 \pi i \mathcal{B}_{ I \cup i, I}}  - \Id) \cdot  \mathscr{F}_{I \cup i}   =  
  \alpha_{I \cup i} \cdot \Theta_{I,i} \cdot B_{I, I \cup i} \cdot  
\mathscr{F}_{I \cup i}   =\\
&= z_i^{-1}  \cdot \alpha_{I  \cup i } \cdot \Theta_{I,i} \cdot \mathscr{F}_I \cdot B_{I, I \cup i} = z_i^{-1} \cdot \eta_{I} ( \alpha_{I  \cup i } \cdot 
\Theta_{I,i}) \cdot B_{I, I \cup i}
\end{align*}
Now, let us compute $M_i  \phi_{I \cup i}^I - \phi_{I \cup i}^I$:
We use the identity $ \mathscr{F}_{I \cup i } \cdot B_{I \cup i, I}= 
 z_i^{-1} \cdot B_{I \cup i, I} \cdot \mathscr{F}_I$
to rearrange $\phi_{I \cup i}^I$. We receive 
\begin{gather*}
 \phi_{I \cup i}^I  
= \alpha_{I \cup i }  \, B_{I \cup i, I} \! \cdot \! \left(  \partial_i^{-1} z_i^{-1} z_{i}^{\mathcal{B}_{I, I \cup i}} \right)  \! \cdot \!
z_{l_{2}}^{\mathcal{B}_{I, I \cup l_{2}}} \cdot 
\ldots \cdot 
 z_{l_{n-|I|}}^{\mathcal{B}_{I, I \cup l_{n-|I|}}} \cdot
 z_{m_1}^{\mathcal{B}_{I, I \setminus m_1}-\Id} \cdot \ldots \cdot 
 z_{m_{|I|}}^{\mathcal{B}_{I, I \setminus m_{|I|}}-\Id} \\
\text{where } \displaystyle \ \ \partial_i^{-1} z_i^{-1} z_{i}^{\mathcal{B}_{I, I \cup i}} = 
\displaystyle \sum_{k=0}^\infty \frac{\mathcal{B}_{I, I \cup i}^k}{k!}  
\frac{\ln(z_i)^{k+1}}{k+1}  =\vcentcolon \varphi_{\mathcal{B}_{I, I \cup i}} (z_i) .
\end{gather*}
\prettyref{cor:phi_A_psi_A} yields
\begin{displaymath}
 M_i \varphi_{\mathcal{B}_{I, I \cup i}} (z_i) - \varphi_{\mathcal{B}_{I, I \cup i}} (z_i) = \biggl(
  \sum_{k=1}^\infty \frac{(2\pi 
i)^k}{k!} \cdot {\mathcal{B}^{k-1}_{I, I \cup i}} \biggr) \cdot z_i^{\mathcal{B}_{I, I \cup i}} .
\end{displaymath}
This gives us 
\begin{align*}
 &M_i  \phi_{I \cup i}^I - \phi_{I \cup i}^I = \\
 &=  \alpha_{I \cup i }  B_{I \cup i, I}     \left( \sum_{k=1}^\infty \frac{(2\pi i)^k}{k!} 
\cdot {\mathcal{B}^{k-1}_{I, I \cup i}} \right) z_i^{\mathcal{B}_{I, I \cup i}}  z_{l_{2}}^{\mathcal{B}_{I, I \cup l_{2}}} 
\ldots 
 z_{l_{n-|I|}}^{\mathcal{B}_{I, I \cup l_{n-|I|}}} 
 z_{m_1}^{\mathcal{B}_{I, I \setminus m_1}-\Id}  \ldots 
 z_{m_{|I|}}^{\mathcal{B}_{I, I \setminus m_{|I|}}-\Id}   = \\
 & = \alpha_{I \cup i } \cdot \Theta_{I,i} \cdot \mathscr{F}_I = \eta_{I } ( \alpha_{I \cup i} \cdot \Theta_{I,i}) .
\end{align*}
Once more, note that these computations are independent of the choice of basis.\qedhere
\end{itemize}
\end{proof}

Now, we have collected all the important pieces for the proof of our Main \prettyref{thm:main_theorem}:
\begin{proof}[Proof of \prettyref{thm:main_theorem}]
We need to examine if the family of isomorphisms from \prettyref{prop:thm_proof_prop_2} is natural: 
The isomorphism 
we gave in part (i) of the proof of \prettyref{prop:thm_proof_prop_2} is natural. So let $ \V = 
(V_J, B_{J \cup 
\set{j}, J}, B_{J, J \cup \set{j}})$ and \mbox{$ \tilde{\V} = (\tilde{V}_J, \tilde{B}_{J \cup \set{j}, J}, \tilde{B}_{J, J \cup 
\set{j}})$} denote two objects in $\Quis_n$ and let $\tau= (h_J)$ denote a morphism from $\V$ to $\tilde{\V}$. We need to check that the diagram

\begin{center}
\begin{tikzpicture}[baseline, >=to, font=\footnotesize]
 \node[matrix] (m) [matrix of math nodes, row sep=1.7em]
    { \{ \oplus_J \phi_I^J \in \oplus_J \Hom_\C ( V_J, \Oh{I,0}) \mid \ldots \}  \\
      \{ \oplus_J \phi_{I \cup \set{i}}^J \in \oplus_J \Hom_\C ( V_J, \Oh{{I \cup \set{i}},0}) \mid \ldots \}   \\ };
\node[matrix] (p) [matrix of math nodes, base right=7.8em of m, row sep=1.7em]
    {  \{ \oplus_J \tilde{\phi}_I^J \in \oplus_J \Hom_\C ( \widetilde{V}_J, \Oh{I,0}) \mid \ldots \} \\
       \{ \oplus_J \tilde{\phi}_{I \cup \set{i}}^J \in \oplus_J \Hom_\C ( \widetilde{V}_J, \Oh{{I \cup \set{i}},0}) \mid \ldots \}  \\ };
\node[matrix] (n) [matrix of math nodes, below=4em of m, row sep=1.7em]
    {  \Hom_\C ( V_I, \C)  \\
      \Hom_\C ( V_{I \cup \set{i}}, \C)  \\ };
\node[matrix] (o) [matrix of math nodes,  below=3.7em of p, row sep=1.7em]
    {  \Hom_\C ( \widetilde{V}_I, \C)  \\
      \Hom_\C ( \widetilde{V}_{I \cup \set{i}}, \C)  \\ };
\draw (m.east |- p.west) edge [<-] node [above, font=\normalsize] {$(\Hom_\C (\tau, \Oh{I,0}))$} (p.west);
\draw[->] (n.north)--(m.south);
\draw ([xshift=5.5em]n.east |- o.west) edge [<-] node [above, font=\normalsize]  {$(h_I^\ast)$} ([xshift=-1em]o.west);
\draw[->] (o.north)--(p.south);
\draw[transform canvas={xshift=-0.5em}, ->]  (m-1-1) -- (m-2-1) node[midway,left] {$\oplus_J \can^J_{I,i}$};
\draw[transform canvas={xshift=0.5em}, ->]  (m-2-1) -- (m-1-1) node[midway,right] {$\oplus_J \var^J_{I,i}$};
\draw[transform canvas={xshift=-0.5em}, ->]  (p-1-1) -- (p-2-1) node[midway,left] {$\oplus_J \widetilde{\can}^J_{I,i}$};
\draw[transform canvas={xshift=0.5em}, ->]  (p-2-1) -- (p-1-1) node[midway,right] {$\oplus_J \widetilde{\var}^J_{I,i}$};
\draw[transform canvas={xshift=-0.5em}, ->]  (n-1-1) -- (n-2-1) node[midway,left] {$B_{I,I \cup \set{i}}^\ast$};
\draw[transform canvas={xshift=0.5em}, ->]  (n-2-1) -- (n-1-1) node[midway,right] {$ \psi ( \mathcal{B}^\ast_{I,I \cup \set{i}}) \circ B_{I \cup \set{i}, I}^\ast $};
\draw[transform canvas={xshift=-0.5em}, ->]  (o-1-1) -- (o-2-1) node[midway,left] {$\tilde{B}_{I,I \cup \set{i}}^\ast$};
\draw[transform canvas={xshift=0.5em}, ->]  (o-2-1) -- (o-1-1) node[midway,right] {$\psi ( \tilde{\mathcal{B}}^\ast_{I,I \cup \set{i}}) \circ \tilde{B}_{I \cup 
\set{i}, I}^\ast $};
\end{tikzpicture}
\end{center}
commutes. The properties indicated by ``$\ldots$'' may be found in part (i) of the proof of \prettyref{prop:thm_proof_prop_2}. The morphisms in the 
horizontal rows are given by
\begin{align*}
h_I^\ast \colon \Hom_\C (\widetilde{V}_I, \C) &\rightarrow \Hom_\C ({V}_I, \C) \, , \ \ \ & \tilde{\alpha}_I &\rightarrow \tilde{\alpha}_I \circ h_I \\
\Hom_\C ( (h_J), \Oh{I,0} ) \colon \textstyle \bigoplus\limits_J \displaystyle  \Hom_\C ( \widetilde{V}_J, \Oh{I,0}) &\rightarrow \textstyle 
\bigoplus\limits_J \displaystyle \Hom_\C ( V_J, \Oh{I,0}) \,  , \ \ \ & \oplus_J \tilde{\alpha}^J_I &\rightarrow \oplus_J ( \tilde{\alpha}^J_I \circ h_J ) .
\end{align*}
The isomorphisms from the lower row into the upper row are given by 
\prettyref{alg:ALG} composed with $(\eta_I)$ and $(\tilde{\eta}_I)$, respectively. The commutativity of the diagram follows now easily using the 
commutativity conditions of the morphism $(h_I)$ with the $B_{\sbt, \sbt}$ and 
$\tilde{B}_{\sbt, \sbt}$-maps. Hence, the diagram of \prettyref{thm:main_theorem} commutes up to a natural isomorphism. The remaining claims 
follow directly from that.
\end{proof}

\addcontentsline{toc}{section}{References}

\vspace{1em}

\rule{100pt}{0.5pt}

\vspace{1em}

\textsc{Stephanie Zapf:} \begin{minipage}[t]{0.79\textwidth} \textit{Lehrstuhl für Algebra und Zahlentheorie\\ Universitätsstraße 14\\ 86159 Augsburg\\
Deutschland}\\
E-mail: \href{mailto:stephanie.zapf@math.uni-augsburg.de}{stephanie.zapf@math.uni-augsburg.de} \hfill \end{minipage}

\begin{thebibliography}{}

\bibitem[Bj{\"o}93]{Bjoerk}
J.-E. Bj{\"o}rk.
\newblock {\em Analytic {$\mathscr{D}$}-modules and applications}, volume 247
  of {\em Mathematics and its applications}.
\newblock Kluwer Academic Publishers, Dordrecht, 1993.

\bibitem[Dim04]{Dimca}
A.~Dimca.
\newblock {\em Sheaves in topology}.
\newblock Universitext. Springer-Verlag, Berlin, 2004.

\bibitem[GGM85a]{GGM1}
A.~Galligo, M.~Granger, and Ph. Maisonobe.
\newblock {${\mathscr{D}}$}-modules et faisceaux pervers dont le support
  singulier est un croisement normal.
\newblock {\em Ann. Inst. Fourier (Grenoble)}, 35(1):1--48, 1985.

\bibitem[GGM85b]{GGM2}
A.~Galligo, M.~Granger, and Ph. Maisonobe.
\newblock {${\mathscr{D}}$}-modules et faisceaux pervers dont le support
  singulier est un croisement normal. {II}.
\newblock {\em Ast\'erisque}, {}(130):240--259, 1985.

\bibitem[GM93]{GM-basiccourse}
M.~Granger and Ph. Maisonobe.
\newblock A basic course on differential modules.
\newblock In {\em \'{E}l\'ements de la th\'eorie des syst\`emes
  diff\'erentiels. {$\mathscr{D}$}-modules coh\'erents et holonomes}, volume~45
  of {\em Travaux en Cours}, pages 103--168. Hermann, Paris, 1993.

\bibitem[HJ91]{Horn-Johnson}
R.~A. Horn and C.~R. Johnson.
\newblock {\em Topics in matrix analysis}.
\newblock Cambridge University Press, Cambridge, 1991.

\bibitem[Kas84]{Kashiwara-RH}
M.~Kashiwara.
\newblock The {R}iemann-{H}ilbert problem for holonomic systems.
\newblock {\em Publ. Res. Inst. Math. Sci.}, 20(2):319--365, 1984.

\bibitem[Kas03]{Kashiwara-microlocalcalc}
M.~Kashiwara.
\newblock {\em {$D$}-modules and microlocal calculus}, volume 217 of {\em
  Translations of mathematical monographs}.
\newblock American Mathematical Society, Providence, R.I., 2003.

\bibitem[KV06]{KV}
S.~Khoroshkin and A.~Varchenko.
\newblock Quiver {$D$}-modules and homology of local systems over an
  arrangement of hyperplanes.
\newblock {\em IMRP Int. Math. Res. Pap.}, 2006:1--116, 2006.
\newblock Art. ID 69590.

\bibitem[Mal91]{Malgrange}
B.~Malgrange.
\newblock {\em Equations diff\'erentielles \`a coefficients polynomiaux},
  volume~96 of {\em Progress in mathematics}.
\newblock Birkh\"auser, Boston, Mass., 1991.

\bibitem[Meb84]{Mebkhout-RH}
Z.~Mebkhout.
\newblock {Une \'equivalence de cat\'egories}, {Une autre \'equivalence de
  cat\'egories}.
\newblock {\em Compos. Math.}, 51(1):51--62, 63--88, 1984.

\bibitem[Ser75]{Serre}
J.-P. Serre.
\newblock {\em Alg\`ebre locale, multiplicit\'es}, volume~11 of {\em Lecture
  notes in mathematics}.
\newblock Springer-Verlag, Berlin, 3rd edition, 1975.

\end{thebibliography}
\end{document}